\documentclass[reqno,11pt]{amsart}
\usepackage{amssymb}
\usepackage{latexsym,array}
\usepackage{amsfonts}
\newcommand{\balpha}{\boldsymbol{\alpha}}
\newcommand{\bbeta}{\boldsymbol{\beta}}
\newcommand{\C}{\mathbb {C}}

\newcommand{\bp}{\boldsymbol{\rho}}

\newcommand{\cX}{\mathcal {X}}

\newcommand{\bH}{\mathbb{H}}

\newcommand{\R}{\mathbb{R}}

\newcommand{\bl}{\boldsymbol{e_\ell}}
\newcommand{\br}{\boldsymbol{e_r}}

\newtheorem{Pa}{Paper}[section]
\newtheorem{theorem}[Pa]{{\bf Theorem}}
\newtheorem{lemma}[Pa]{{\bf Lemma}}
\newtheorem{algorithm}[Pa]{{\bf Algorithm}}

\newtheorem{corollary}[Pa]{{\bf Corollary}}

\newtheorem{remark}[Pa]{{\bf Remark}}

\author[V. Bolotnikov]{Vladimir Bolotnikov}
\address{Williamsburg}
\title[Sylvester equation]{On the Sylvester matrix equation over quaternions}
\begin{document}
\begin{abstract}
The Sylvester equation $AX-XB=C$ is considered in the setting of quaternion matrices.
Conditions that are necessary and sufficient for the existence of a unique solution 
are well-known. We study the complementary case where the equation either has infinitely
many solutions or does not have solutions at all.
Special attention is given to the case where $A$ and $B$ are respectively, lower and 
upper triangular two-diagonal matrices (in particular, if $A$ and $B$ are Jordan blocks).
\end{abstract}

\maketitle 
\section{Introduction}
\setcounter{equation}{0}

We start with the complex case: given complex matrices $A\in\C^{n\times n}$, $B\in\C^{m\times m}$, 
$C\in\C^{n\times m}$, the Sylvester equation 
\begin{equation}
AX-XB=C
\label{1.1}
\end{equation}
has a unique solution $X=\left[x_{ij}\right]\in\C^{n\times m}$
if and only if the spectrums $\sigma(A)$ and $\sigma(B)$ are disjoint.
The result was established by Sylvester \cite{sylv} via representing  
\eqref{1.1} in the equivalent form $G{\bf x}={\bf c}$ where $G=A\otimes I_m-I_n\otimes B^\top$
and where ${\bf x}$, ${\bf c}$ are the columns constructed from the entries 
of $X$ and $C$ by 
\begin{equation}
{\bf x}={\rm Col}_{1\le i\le
n}\left({\rm Col}_{1\le i\le m}x_{ij}\right),\quad
{\bf c}={\rm Col}_{1\le i\le
n}\left({\rm Col}_{1\le i\le m}c_{ij}\right).
\label{1.2}
\end{equation}
Thus, the equation \eqref{1.1} has a unique solution (for any $C\in\C^{n\times
m}$) if and only if the matrix $G$ is invertible (equivalently,  
$\sigma(A)\cap\sigma(B)=\emptyset$), in which case the unique solution $X$ is recovered from the 
column $G^{-1}{\bf c}$. The infinite-dimensional extension of the Sylvester's theorem 
as well as the integral formula for the unique solution $X$ of \eqref{1.1} is due to
Rosenblum \cite{rosen}. Some other explicit formulas for $X\in\C^{n\times m}$ (in terms of $A$, $B$ and 
$C$ rather than their entries) can be found in the survey \cite{lanc}.
If $\sigma(A)\cap\sigma(B)\not=\emptyset$, then \eqref{1.1} has a 
solution if and only if ${\rm rank}G={\rm rank}\begin{bmatrix}G & {\bf c}\end{bmatrix}$,
while the homogeneous equation $AX=XB$ has $d$ linearly independent solutions where the integer $d=\dim 
{\rm Nul}(G)$ can be expressed in terms of invariant factors of  matrix pencils $A$ and $B$ 
\cite{ces, fro}). 

\smallskip

Let  $\mathcal J_n(\alpha)$ denote the $n\times n$ lower
triangular Jordan block with $\alpha$ on the main diagonal:
\begin{equation}
\mathcal J_n(\alpha)=\alpha I_n+F_n,\quad\mbox{where}\quad F_n=\left[\delta_{i,j+1}\right]_{i,j=1}^{n}
\label{1.3}
\end{equation}
and where $\delta_{i,j}$ is the Kronecker symbol. If $A$ and $B$ are Jordan blocks 
$A=\mathcal J_n(\alpha)$ and 
$B=\mathcal J^\top_{m}(\beta)$, the equation \eqref{1.1} simplifies to
\begin{equation}
(\alpha-\beta)X=C-F_n X+XF_m^\top.
\label{1.4}
\end{equation}
If $\alpha\neq \beta$, we iterate the latter equality $n+m-1$ times and arrive at
\begin{equation}
X=\sum_{k=0}^{n+m-2}(\alpha-\beta)^{-k-1}\sum_{i+j=k}(-1)^{k+i}\binom{k}{i}F_n^iC(F^\top_m)^j.
\label{1.5}
\end{equation}
The formula \eqref{1.5} appeared in \cite{ma} and in a slightly different form, in \cite{rutherford}.
If $\alpha=\beta$ and $n\ge m$ (this case was considered in \cite{ma}), the equation \eqref{1.4} can be
written entry-wise as follows:
$$
0=c_{1,1}, \quad x_{i,1}=c_{i+1,1}, \quad x_{1,j}=-c_{1,j+1},
$$
\begin{equation}
x_{i,j+1}-x_{i+1,j}=c_{i+1,j+1}\quad\mbox{for}\quad 1\le i< n; \; \; 1\le j< m.
\label{1.6}
\end{equation}
Then one can see that if the system \eqref{1.6} has a solution, then necessarily
\begin{equation}
\sum_{i=1}^{k-1}c_{i,k-i}=0\quad\mbox{for}\quad k=2,\ldots,n.
\label{1.7}
\end{equation}
On the other hand if conditions \eqref{1.7} are met, then for any choice of
fixed $x_{n,1},\ldots,x_{n,m}$, the system \eqref{1.5} has a unique solution.
In other words, the bottom row in $X$ serves as a free and independent parameter in
the parametrization of all solutions $X$ of the Sylvester equation \eqref{1.4}.

\smallskip

In general, one may reduce given matrices $A$ and $B$ to their lower and upper Jordan forms
(with say, $k$ and $\ell$ Jordan cells, respectively) subsequently splitting the equation \eqref{1.1} to 
$k\ell$ Sylvester equations of the form \eqref{1.4}. 

\smallskip

In this paper, we will focus on the equation \eqref{1.1} over quaternions:
\begin{equation}
AX-XB=C,\quad\mbox{where}\quad A\in\bH^{n\times n}, \; B\in\bH^{m\times m}, \; C\in\bH^{n\times m}.
\label{1.10}   
\end{equation}
Since multiplication in $\bH$ is non-commutative, 
the equation \eqref{1.10} is not completely trivial even in the 
scalar case. The following result back to Hamilton (see e.g., \cite[p. 123]{tait}).
\begin{theorem}
Given $\alpha,\beta,c\in\bH$, the equation $\; \alpha x-x\beta =c \; $ has a unique solution if and only 
if 
\begin{equation}
P_{\alpha,\beta}:=|\alpha|^2-(\alpha+\overline\alpha)\beta+\beta^2\neq 0
\label{1.11}  
\end{equation}
and this unique solution equals $x=(\overline{\alpha}c-c\beta)P_{\alpha,\beta}^{-1}$.
\label{T:1.1}
\end{theorem}
Taking the advantage of complex representations for quaternion matrices \cite{lee}, it        
is always possible to reduce \eqref{1.10} to certain complex Sylvester equation producing     
in particular, the uniqueness criterion: {\em the equation \eqref{1.10} has a unique solution if and only 
if the right spectrums of $A$ and $B$ are disjoint}; see \cite{huang, song} and Theorem \ref{T:3.1} below.
However, further results obtained on this way and briefly surveyed in Section 3
refer, to some extent, to complex representations of matrices $A$ and $B$ rather
than the matrices themselves; see e.g., formula \eqref{2.9u} for the unique solution.

\smallskip

Our contribution here are several explicit formulas for the unique
solution given exclusively in terms of the original matrices $A$, $B$, $C$ in the cases where 
(1) $A$ and $B$ are Jordan blocks (Theorem \ref{T:3.2} presents the quaternion analog of formula 
\eqref{1.5}), (2) $A$ is two-diagonal (Theorems \ref{T:3.2a} and \ref{T:3.8}), and (3) $A$ is 
lower-triangular (Theorem \ref{T:3.6}). Making use of canonical Jordan forms for $A$ and $B$
(see \cite{wieg} and Theorem \ref{T:2.3} below) one can reduce the general case to the one
where $A$ and $B$ are Jordan cells.

\smallskip

The core of the paper is the study of the singular case (the right 
spectrums of $A$ and $B$ are not disjoint). Special attention is given to the case where 
$A$ and $B$ are two-diagonal matrices (see formulas \eqref{2.20}). 
As will be explained in Section 2.4, Sylvester equations with $A$ and $B$ of this form 
arise in the context of polynomial interpolation over quaternions and the results on such
special Sylvester equations are needed to explicitly describe quasi-ideals in the ring of 
quaternion polynomials. In Section 4 we present necessary and 
sufficient conditions for the equation to have a solution (Theorem \ref{T:4.4}) which become 
more transparent if $A$ and $B$ are Jordan blocks (Theorem \ref{T:4.6} presents the analog of conditions 
\eqref{1.7} from \cite{ma}). Also in Section 4, we present an algorithm for constructing a solution
to the (solvable) equation \eqref{1.10}. In Section 5, we parametrize the solution set of the homogeneous 
equation $AX-XB=0$; the parametrization contains $\min(m,n)$ free independent parameters, each one of 
which varies in a two-dimensional real subspace of $\bH$. In case $A=\mathcal I_n(\alpha)$ and $B=\mathcal 
I_m(\beta)$ are Jordan blocks, the general solution of the homogeneous equation is a ``triangular" Hankel 
matrix all entries of which satisfy the homogeneous scalar Sylvester equation $\alpha x-x\beta=0$ 
(Corollary \ref{C:5.2}). 

\section{Preliminaries}
\setcounter{equation}{0}

In this section we collect basic facts to make presentation self-contained.
We first fix notation and terminology. By $\bH$ we denote the skew field of quaternions
$\alpha=x_0+{\bf i}x_1+{\bf j}x_2+{\bf k}x_3$ where  $x_0,x_1,x_2,x_3\in\mathbb R$
and where ${\bf i}, {\bf j}, {\bf k}$ are the imaginary units commuting with $\R$ and
satisfying ${\bf i}^2={\bf j}^2={\bf k}^2={\bf ijk}=-1$. For $\alpha\in\bH$ as above,
its real and imaginary parts, the quaternion conjugate and the absolute value
are defined as ${\rm Re}(\alpha)=x_0$, ${\rm Im}(\alpha)={\bf i}x_1+{\bf j}x_2+{\bf k}x_3$,
$\overline \alpha={\rm Re}(\alpha)-{\rm Im}(\alpha)$ and
$|\alpha|^2=\alpha\overline{\alpha}=|{\rm Re}(\alpha)|^2+|{\rm Im}(\alpha)|^2$,
respectively.  Two quaternions $\alpha$ and $\beta$ are called {\em equivalent} (conjugate
to each other) if $\alpha=h^{-1}\beta h$ for some nonzero $h\in\mathbb H$; in notation,
$\alpha\sim\beta$. It turns out (see e.g., \cite{brenner}) that
\begin{equation}
\alpha\sim\beta\quad\mbox{if and only if}\quad {\rm Re}(\alpha) ={\rm Re}(\beta) \;
\mbox{and} \; |\alpha|=|\beta|,
\label{2.1}
\end{equation}
so that the {\em conjugacy class} of a given $\alpha\in\mathbb H$ form a $2$-sphere (of
radius $|{\rm Im}(\alpha)|$ around ${\rm Re}(\alpha)$) which will be denoted  by $[\alpha]$.
It is clear that $[\alpha]=\{\alpha\}$ if and only if $\alpha\in\R$. 

\smallskip

A finite ordered collection ${\balpha}=(\alpha_1,\ldots,\alpha_n)$ will be called a {\em spherical 
chain} (of length $n$) if
\begin{equation}
\alpha_1\sim \alpha_2\sim\ldots\sim\alpha_k\quad\mbox{and}\quad \alpha_{j+1}\neq
\overline{\alpha}_j\quad\mbox{for}\quad j=1,\ldots,n-1.
\label{2.2}
\end{equation}
The latter notion is essentially non-commutative: a spherical chain 
$\balpha=(\alpha_1,\ldots,\alpha_n)$ consisting of commuting elements necessarily belongs to the 
set $\{\alpha_1,\overline{\alpha}_1\}$ which
together with inequality in \eqref{2.2} implies that all elements in $\balpha$ are the same:
\begin{equation}
\balpha=(\alpha,\alpha,\ldots,\alpha),\qquad \alpha\in\bH.
\label{2.3}
\end{equation}
\subsection{Quaternion matrices}
We denote by $\bH^{n\times m}$ the space of $n\times m$ matrices with quaternion entries.
The definitions of the transpose matrix $A^\top$, the quaternion-conjugate matrix $\overline{A}$
and the adjoint matrix $A^*$ are the same as in the complex case. 

\smallskip

An element $\alpha\in\bH$ is called a (right) eigenvalue of the matrix $A\in\bH^{n\times n}$
if $A{\bf x}={\bf x}\alpha$ for some nonzero ${\bf x}\in{\bH}^{n\times 1}$. In this case,
for any $\beta=h^{-1}\alpha h\sim \alpha$ we also have $A{\bf x}h=  {\bf x}hh^{-1}\alpha h={\bf x}h\beta$
and hence, any element in the conjugacy class $[\alpha]$ is a right eigenvalue of $A$. Therefore,
the right spectrum $\sigma_{\bf r}(A)$ is the union of disjoint conjugacy classes (some of which 
may be real singletons). The studies of right eigenvalues and canonical forms for quaternion 
matrices were carried out in \cite{brenner, lee, wieg}. In particular, it was shown in \cite{wieg}
that any square quaternion matrix is similar to a complex matrix in Jordan form.
\begin{theorem}
For every $A\in\bH^{n\times n}$, there is an invertible $S\in\bH^{n\times n}$ such that 
$S^{-1}AS={\displaystyle\bigoplus_{i=1}^k \mathcal J_{n_i}(\alpha_i)}$, where 
$\alpha_i\in\C$ and $\; {\rm Im}\alpha_i\ge 0\; $ for $i=1,\ldots,k$.
\label{T:2.3}
\end{theorem}
Making use of Theorem \ref{T:2.3} one can reduce $A$ and $B$ to their lower and upper Jordan forms
$$
S^{-1}AS=\bigoplus_{i=1}^k \mathcal J_{n_i}(\alpha_i),\quad T^{-1}BT=\bigoplus_{j=1}^\ell \mathcal
J^\top_{m_j}(\beta_j)
$$
and conformally decompose $S^{-1}CT=[C_{ij}]$ and $S^{-1}XT=[X_{ij}]$ to see that
\eqref{1.10} splits into $k\ell$ Sylvester equations
\begin{equation}
\mathcal J_{n_i}(\alpha_i)X_{ij}-X_{ij}\mathcal J^\top_{m_j}(\beta_j)=C_{ij}.
\label{1.9}   
\end{equation}
This reduction suggests to study the ``basic" case where $A$ and $B$ are Jordan blocks, i.e.,
to get the quaternion analogs of the formula \eqref{1.5} and conditions \eqref{1.7}.
The defficiency of this approach is that most of explicit formulas will rely on  
similarity matrices $S$ and $T$. 
 \subsection{Complex representations}
Since each $\alpha\in\bH$ admits a unique representation of the form $\alpha=\alpha_1+\alpha_2{\bf j}$
with $\alpha_1,\alpha_2\in\C$, any matrix $A\in\bH^{n\times m}$ can be written uniquely as
$A=A_1+A_2{\bf j}$ with $A_1, \, A_2\in\C^{n\times m}$. The map
\begin{equation}
A=A_1+A_2{\bf j}\mapsto \varphi(A)=\begin{bmatrix}A_1 & A_2 \\ -\overline{A}_2 & \overline{A}_1\end{bmatrix}
\label{2.4}
\end{equation}
that associates to each quaternion matrix its {\em complex representation} was introduced in \cite{lee}.
It is additive and multiplicative in the sense that 
\begin{equation}
\varphi(A+B)=\varphi(A)+\varphi(B)\quad\mbox{and}\quad\varphi(AB)=\varphi(A)\varphi(B)
\label{2.5}
\end{equation}
for rectangular quaternion matrices of appropriate sizes. It was shown in \cite{lee} that if $\lambda\in\C$ is 
an eigenvalue of $\phi(A)$, then any element
from the $2$-sphere $[\lambda]\subset\bH$ is a right eigenvalue of $A$ and that all right eigenvalues of $A$
arise in this way (this fact implies that the right spectrum of an $n\times n$ matrix
is the union of at most $n$ disjoint conjugacy classes).
Therefore, for matrices $A\in\bH^{n\times n}$ and $B\in\bH^{m\times m}$, the following conditions
are equivalent:
\begin{equation}
\sigma(\varphi(A))\cap\sigma(\varphi(B))=\emptyset\quad\Longleftrightarrow\quad
\sigma_{\bf r}(A)\cap \sigma_{\bf r}(B)=\emptyset. 
\label{2.6}
\end{equation}
Observing that ${\bf j}D=\overline{D}{\bf j}$ for any complex matrix $D$ we define the map 
\begin{align}
\psi: \; Y=\begin{bmatrix}Y_{11} & Y_{12} \\ Y_{21} & Y_{22}\end{bmatrix}\mapsto&  
\frac{1}{2}\begin{bmatrix}I_n & -{\bf j}I_n\end{bmatrix}Y\begin{bmatrix}I_m \\ {\bf
j}I_m\end{bmatrix}\label{2.7}\\
&=\frac{Y_{11}+\overline{Y}_{22}}{2}+\frac{Y_{12}-\overline{Y}_{21}}{2} \, {\bf j}\notag
\end{align}
assigning to each matrix $Y\in\C^{2n\times 2m}$ a quaternion matrix constructed from 
the blocks $Y_{ij}\in\C^{n\times m}$. It is easily verified that $\psi$ is the left inverse of 
$\phi$: 
\begin{equation}
\psi(\phi(A))=A\quad\mbox{for any}\quad A\in\bH^{n\times m}.
\label{2.8}   
\end{equation}
Besides, $\psi$ is additive and, although not multiplicative, the equalities
\begin{equation}
\begin{array}{ll}
\psi(\phi(A)Y)&=\psi(\phi(A))\psi(Y)=A\psi(Y),\\ [1mm]
\psi(Y\phi(B))&=\psi(Y)\psi(\phi(B))=\psi(Y)B\end{array}
\label{2.9}
\end{equation}
hold for any $A\in\bH^{n\times m}$, $B\in\bH^{p\times q}$ and $Y\in\C^{2m\times 2p}$.

\subsection{Quaternion polynomials}
Let $\bH[z]$ denote the ring of polynomials in one formal variable $z$ which commutes with
quaternionic coefficients. The ring operations in $\bH[z]$ are defined as in the commutative
case, but as multiplication in $\bH$ is not commutative, multiplication in $\bH[z]$ is not
commutative as well. For any $\alpha\in\bH$, we define left and right evaluation of
$f$ at $\alpha$ by
\begin{equation}
f^{\bl}(\alpha)=\sum_{j=0}^k\alpha^j f_j\quad\mbox{and}\quad
f^{\br}(\alpha)=\sum_{j=0}^k f_j\alpha^j\quad\mbox{if}\quad f(z)=\sum_{j=0}^k z^j f_j.
\label{2.10}
\end{equation}
The formulas make sense for matrix valued polynomials $f\in\bH^{n\times m}[z]$ and extend 
to square matrices by letting
\begin{equation}
f^{\bl}(A)=\sum_{j=0}^kA^j f_j,\quad
f^{\br}(B)=\sum_{j=0}^k f_j B^j\quad\mbox{if}\quad A\in\bH^{n\times n}, \; B\in \bH^{n\times n}.
\label{2.11}
\end{equation}
An element $\alpha\in\bH$ is called a left (right) zero of $f$ if $f^{\bl}(\alpha)=0$
(respectively, $f^{\br}(\alpha)=0$). For polynomials with real coefficients,
left and right evaluations (and therefore, the notions of left and right zeros) coincide.
The {\em characteristic polynomial} of a non-real conjugacy class $[\alpha]\subset\bH$
is defined by
\begin{equation}
\cX_{[\alpha]}(z)=(z-\alpha)(z-\overline{\alpha})=z^2-z(\alpha+\overline{\alpha})+|\alpha|^2;
\label{2.12}
\end{equation}
it follows from characterization \eqref{2.1} that formula \eqref{2.12}
does not depend on the choice of $\alpha\in [\alpha]$. Since $\cX_{[\alpha]}$ is the polynomial of the 
minimally possible 
degree such that its zero set (left and right, as $\cX\in{\mathbb R}[z]$) coincides with $[\alpha]$, it is also 
called the {\em minimal polynomial} of $[\alpha]$. Observe that the matrix $\cX_{[\alpha]}(B)$ is 
invertible
if and only if $[\alpha]\cap\sigma_{\bf r}(B)=\emptyset$.

\smallskip

Since the division algorithm holds in $\bH[z]$ on either side (see e.g., \cite{wed}), 
any (left or right) ideal in $\bH[z]$ is principal.  We will use notation $\langle h\rangle_{\bf r}$ and 
$\langle h\rangle_{\boldsymbol\ell}$ for respectively the right and the left ideal generated by $h$. An ideal is 
maximal if and only if it is generated by the polynomial 
\begin{equation}
\bp_\alpha(z)=z-\alpha,\qquad \alpha\in\bH
\label{2.13}  
\end{equation}
and it follows from respective (left and right) division algorithms that 
\begin{equation}
f\in \langle \bp_\alpha\rangle_{\bf r} \; \Leftrightarrow \; f^{\bl}(\alpha)=0,\qquad
f\in \langle \bp_\alpha\rangle_{\boldsymbol\ell} \; \Leftrightarrow \; f^{\br}(\alpha)=0.
\label{2.9e}  
\end{equation}
A left (right) ideal is called {\em irreducible} if it is not contained properly in two distinct left 
(right) ideals which occurs if and only if it is generated by a polynomial $p$ of the form 
$p=\bp_{\alpha_1}\bp_{\alpha_2}\cdots \bp_{\alpha_n}$
for some spherical chain $\balpha=(\alpha_1,\ldots,\alpha_n)$ (see e.g., \cite{bolalg}).

\subsection{Sylvester equations and interpolation by polynomials} 
The equation \eqref{1.10} arises naturally in the context of the following interpolation problem:
{\em given polynomials $p$, $\widetilde{p}$, $g$, $\widetilde{g}$, find a polynomial $f\in\bH[z]$ such
that
\begin{equation}
f-g \in \langle p\rangle_{\bf r}\quad\mbox{and}\quad f-\widetilde{g} \in\langle \widetilde{p}\rangle_{\boldsymbol\ell}.
\label{2.15}
\end{equation}}
If $p=\bp_\alpha$ and $q=\bp_\beta$, then due to equivalences \eqref{2.9e}, conditions 
\eqref{2.15} can be written in terms of left and right evaluations as 
\begin{equation}
f^{\bl}(\alpha)=\gamma\quad \mbox{and}\quad f^{\br}(\beta)=\delta
\label{2.8b}
\end{equation}
where $\gamma=h^{\bl}(\alpha)$ and $\delta=g^{\br}(\beta)$.
The following result appears in \cite{bol}.
\begin{theorem}
There is a polynomial $f\in\bH[z]$ satisfying conditions \eqref{2.8b}
if and only if the Sylvester equation
\begin{equation}
\alpha x-x\beta =\gamma-\delta
\label{2.9b}
\end{equation}
has a solution. If this is the case, all polynomials $f\in\bH[z]$ satisfying conditions \eqref{2.8b}
are parametrized by the formula
\begin{equation}
f=\gamma+\bp_\alpha x+\bp_\alpha h\bp_\beta \quad(=\delta+x\bp_{\beta}+\bp_\alpha h\bp_\beta)
\label{2.10b}
\end{equation}
when $x$ runs through the solution set of \eqref{2.9b} and $h$ varies in $\bH[z]$.
\label{T:2.1}
\end{theorem}
Note that letting $h\equiv 0$ in \eqref{2.10b} gives all linear solutions to the problem \eqref{2.8b}
while letting $\gamma=\delta=0$ leads us to the explicit description
of the intersection $\langle \bp_\beta\rangle_{\boldsymbol\ell}\cap \langle \bp_\alpha\rangle_{\bf r}$
of two maximal ideals in $\bH[z]$. 

\smallskip

Let us now assume that the ideals $\langle p\rangle_{\bf r}$ and $\langle q\rangle_{\boldsymbol\ell}$ in 
\eqref{2.15} 
are irreducible, i.e., that $p$ and $q$ are of the form  
\begin{equation}
p=\bp_{\alpha_1}\bp_{\alpha_2}\cdots \bp_{\alpha_n},\quad
q=\bp_{\beta_m}\bp_{\beta_{m-1}}\cdots \bp_{\beta_1}
\label{2.19}
\end{equation}
for some spherical chains where $\balpha=(\alpha_1,\ldots,\alpha_n)$ and $\bbeta=(\beta_1,\ldots,\beta_m)$.
Let us introduce the matrices
\begin{equation}
\mathcal J_{{\balpha}}=\left[\begin{array}{ccccc} \alpha_1&0&\ldots&&0\\ 1
&\alpha_2&0&&\\ 0&1&\ddots&\ddots&\vdots\\
\vdots&\ddots&\ddots&\ddots&0\\ 0&\ldots &0&1&\alpha_n
\end{array}\right],\quad 
\mathcal J_{{\bbeta}}=\left[\begin{array}{ccccc} \beta_1&0&\ldots&&0\\ 1
&\beta_2&0&&\\ 0&1&\ddots&\ddots&\vdots\\
\vdots&\ddots&\ddots&\ddots&0\\ 0&\ldots &0&1&\beta_m
\end{array}\right]
\label{2.20}  
\end{equation}
and let $E_n\in\bH^{n\times 1}$ be the column with the top entry equal one and all other entries equal zero.
\begin{theorem}
Let $p$ and $q$ be as in  \eqref{2.19} and, given polynomials 
$g(z)=\sum g_jz^j$ and $\widetilde{g}(z)=\sum \widetilde{g}_kz^k$
with $\deg g<\deg p$ and $\deg \widetilde{g}<\deg \widetilde{p}$, let 
\begin{equation}
C=\sum_j \mathcal J_{{\balpha}}^jE_n g_jE_m^\top-\sum_k E_n \widetilde{g}_k E_m^\top \mathcal J_{{\bbeta}}^\top.
\label{2.12uu}  
\end{equation}
Then there exists a polynomial $f\in\bH[z]$ satisfying conditions \eqref{2.15}
if and only if the Sylvester equation
\begin{equation}
\mathcal J_{{\balpha}}X-X\mathcal J_{{\bbeta}}^\top=C
\label{2.9bb}
\end{equation}  
has a solution $X=\left[x_{ij}\right]\in\bH^{n\times m}$. If this is the case, all polynomials $f\in\bH[z]$ satisfying 
conditions \eqref{2.15} are parametrized by the formula
\begin{equation}
f=g+\sum_{j=1}^m p\cdot x_{n,1}\cdot \bp_{\beta_j}\bp_{\beta_{j-1}}\cdots \bp_{\beta_1} +p\cdot h\cdot \widetilde{p}
\label{2.10bb}
\end{equation}
when $\begin{bmatrix}x_{n,1}&\ldots&x_{n,m}\end{bmatrix}$ is the bottom row of an $X$ satisfying \eqref{2.9bb}
and $h$ varies in $\bH[z]$. The polynomial $f$ in \eqref{2.10bb} can be written in the form 
\begin{equation}
f=\widetilde{g}+\sum_{k=1}^n \bp_{\alpha_1}\bp_{\alpha_2}\cdots \bp_{\alpha_k}\cdot x_{k,m}\cdot \widetilde{p} 
+p\cdot h\cdot \widetilde{p}.
\label{2.10c}
\end{equation}
\label{T:2.2}
\end{theorem} 
The latter theorem is an extension of Theorem \ref{T:2.1}. Formulas \eqref{2.10b} suggest to conjecture
that in more general setting of Theorem \ref{T:2.2}, all solutions $f$ to the interpolation problem
should be paprametrized by the ``whole" solutions $X$ to the equation \eqref{2.9bb} rather than by their  
bottom rows (in formula \eqref{2.10bb}) or the rightmost columns (in formula \eqref{2.10c}). However, 
there is no contradiction here: as we will see in Section 4, any solution $X$ to the equation 
\eqref{2.9bb} is completely determined by its bottom row or its rightmost column.

\smallskip

Letting $h\equiv 0$ in \eqref{2.10bb} (or in 
\eqref{2.10c}) gives all solutions $f$ to the  problem \eqref{2.15}  with $\deg f<\deg p+\deg \widetilde{p}$ 
while letting $g=\widetilde{g}\equiv0$ leads to fairly explicit description
of the intersection $\langle p\rangle_{\boldsymbol\ell}\cap \langle \widetilde{p}\rangle_{\bf r}$
of two irreducible ideals in $\bH[z]$. The case of general polynomials $p$ and $\widetilde{p}$ in \eqref{2.15} can 
be 
reduced to the irreducible case upon making use of the Primary Ideal Decomposition Theorem and gives rise to the 
Sylvester equation \eqref{1.10} with $A$ and $B$ in the block diagonal form with all diagonal blocks of the from 
\eqref{2.20}. The formulation of Theorem \ref{T:2.2} is presented here to explain our interest in the 
Sylvester equation with $A$ and $B$ of the form \eqref{2.20}. The proof of the theorem along with 
reformulations of conditions \eqref{2.15} in terms of evaluation functionals is given in \cite{bolalg1}.

\section{Quaternion Sylvester equation: the regular case} 
\setcounter{equation}{0}
Applying the map \eqref{2.4}
to the equation \eqref{1.10} we observe that for any solution $X\in\bH^{n\times m}$ of \eqref{1.10},
the matrix $Y=\varphi(X)\in\C^{2n\times 2m}$ solves the complex Sylvester equation
\begin{equation}
\varphi(A)Y-Y\varphi(B)=\varphi(C).
\label{3.1}   
\end{equation}
By the Sylvester's criterion and due to equivalence \eqref{2.6}, the latter equation 
has a unique solution if and only if $\sigma_{\bf r}(A)\cap \sigma_{\bf 
r}(B)=\emptyset$. If this is the case, we apply the map \eqref{2.7} to the equation \eqref{3.1}.
Due to properties \eqref{2.8}, \eqref{2.9} of $\psi$ we get $A\psi(Y)-\psi(Y)B=C$ so that
$X=\psi(Y)$ is a solution to \eqref{1.10}. It is a unique solution since $\phi$ is injective, so that 
distinct solutions to \eqref{1.10} would have given rise to distinct solutions of \eqref{3.1}.
We thus arrive at the following result (see e.g., \cite{huang}):
\begin{theorem}
Equation \eqref{1.10} has a unique solution (for every $C\in\bH^{n\times m}$) if and only if $\sigma_{\bf 
r}(A)\cap \sigma_{\bf r}(B)=\emptyset$.
\label{T:3.1}
\end{theorem}
As we have seen, the unique solution of the equation
\eqref{1.10} is necessarily of the form $X=\psi(Y)$ where $Y$ is the unique solution to the complex
Sylvester equation \eqref{3.1}. Taking any formula for $Y$ available in literature, 
one gets a formula for $X$ by letting $X=\psi(Y)$. For example, if we denote by $\mu: \, \C^{m\times n}\to
\C^{mn\times 1}$ the bijection assigning to the matrix $X=[x_{ij}]$ the 
the column ${\bf x}$ as in \eqref{1.2}, the formula for the unique solution $X$ of the equation
\eqref{1.10} suggested by the  original Sylvester's approach \cite{sylv} is 
\begin{equation}
X=\psi(\mu^{-1} ( (\varphi(A)\otimes I_{2m}-I_{2n}\otimes \varphi(B))^{-1}\mu(\phi(C))   )).
\label{2.m}
\end{equation}
The formula with minimal references to complex representations (extending the Jameson's result 
\cite{jameson}) has been established in \cite{huang, song}:
\begin{equation}
X=\left(\sum_{k=1}^{2n}\sum_{j=0}^{k-1}a_kA^jCB^{k-j-1}\right)\left(\sum_{j=0}^{2n} a_jB^j\right)^{-1},
\label{2.9u}
\end{equation}
where $a_0,\ldots,a_{2n}\in\C$ are the coefficients of the polynomial 
$$
\det (\lambda I_{2n}-\phi(A))=a_0+a_1\lambda+\ldots +a_{2n}\lambda^{2n}.
$$
In the rest of the section we examine how far one can advance making no use of complex representations
of quaternion matrices. We start with a very special case 
where $A$ and $B$ are Jordan blocks and establish the quaternion analog of the formula \eqref{1.5}.
On the other hand, this result generalizes Theorem \ref{T:1.1}. 
\begin{theorem}
Let $A$ and $B$ be Jordan blocks $A=\mathcal J_n(\alpha)$ and
$B=\mathcal J^\top_{m}(\beta)$ for some $\alpha\not\sim\beta$. Then the equation 
\eqref{1.10} has a unique solution 
\begin{equation}
X=\sum_{k=0}^{n+m-2}\left(\sum_{i=0}^{k+1}(-1)^{k+i}\binom{k+1}{i}\overline{\alpha}^{k-i+1}
M_k\beta^i\right)P_{\alpha,\beta}^{-k-1},
\label{3.4}   
\end{equation}
where $P_{\alpha,\beta}$ is given by \eqref{1.11} and 
$$
M_k=\sum_{\ell=0}^k(-1)^{\ell}\binom{k}{\ell}F_n^{k-\ell} C(F_m^\top)^\ell\quad\mbox{for}\quad k=0,\ldots,n+m-2.
$$
\label{T:3.2}
\end{theorem} 
\begin{proof} 
Observe that
$P_{\alpha,\beta}=\cX_{[\alpha]}(\beta)$ (the value of the characteristic polynomial \eqref{2.12}
at $\beta$) and thus, $P_{\alpha,\beta}\neq 0$ if and only if $\alpha\not\sim\beta$.
As $A=\alpha I_n+F_n$ and  $B=\beta I_n +F_m^\top$ (see \eqref{1.3}), the equation 
\eqref{1.10} takes the form 
\begin{equation}
\alpha X-X\beta=C-F_n X+XF_m^\top.
\label{3.5}
\end{equation}
We now subtract the latter equation multiplied by $\beta$ on the right from the same equation multiplied by 
$\overline{\alpha}$ on the left:
$$
|\alpha|^2X-(\overline{\alpha}+\alpha)X\beta+X\beta^2=\overline{\alpha}(C-F_n X+XF_m^\top)-
(C-F_n X+XF_m^\top)\beta.
$$
Since $|\alpha|^2$ and $(\alpha+\overline{\alpha})$ are real and therefore commute with all quaternions, 
the latter equality can be written as 
\begin{equation}
XP_{\alpha,\beta}=\overline{\alpha}(C-F_n X+XF_m^\top)-
(C-F_n X+XF_m^\top)\beta,
\label{2.12u}
\end{equation}
in view of \eqref{1.11}. Thus, equation \eqref{2.12u} follows from \eqref{3.5}.
On the other hand, we may subtract the equation \eqref{2.12u} multiplied by $\beta$ on the right from the 
same equation multiplied by $\alpha$ on the left:
$$
\alpha XP_{\alpha,\beta}-XP_{\alpha,\beta}\beta=(C-F_n X+XF_m^\top)P_{\alpha,\beta}.
$$
Since $P_{\alpha,\beta}\neq 0$ commutes with $\beta$, we may cancel $P_{\alpha,\beta}$ in the latter equation
arriving at \eqref{3.5}. Therefore, equations \eqref{3.5} and \eqref{2.12u} are equivalent. 
Since $P_{\alpha,\beta}\beta=\beta P_{\alpha,\beta}$ and since the matrices 
$F_n$ and $F_m$ are real, we can iterate the equality \eqref{2.12u} as follows:
\begin{align*}
XP_{\alpha,\beta}^{r}=&(\overline{\alpha}C-C\beta)P_{\alpha,\beta}^{r-1}-
\overline{\alpha}F_nXP_{\alpha,\beta}^{r-1}
+\overline{\alpha}XP_{\alpha,\beta}^{r-1}F_m^\top\\
&+F_n XP_{\alpha,\beta}^{r-1}\beta-XP_{\alpha,\beta}^{r-1}F_m^\top \beta\\
=&(\overline{\alpha}C-C\beta)P_{\alpha,\beta}^{r-1}
-\overline{\alpha}^2(F_nC-CF_m^\top)P_{\alpha,\beta}^{r-2}\\
&+2\overline{\alpha}(F_nC-CF_m^\top)\beta P_{\alpha,\beta}^{r-2}-(F_nC-CF_m^\top)\beta^2P_{\alpha,\beta}^{r-2}\\
&+F_n(
\overline{\alpha}^2(F_n X-XF_m^\top)-\overline{\alpha}(F_n X-XF_m^\top)\beta)P_{\alpha,\beta}^{r-2}\\
&-(\overline{\alpha}^2(F_n X-XF_m^\top)-\overline{\alpha}(F_n X-XF_m^\top)\beta)  P_{\alpha,\beta}^{r-2}F_m^\top\\
&-F_n (\overline{\alpha}(F_n X-XF_m^\top)\beta-(F_n X-XF_m^\top)\beta^2)P_{\alpha,\beta}^{r-2}\\
&+(\overline{\alpha}(F_n X-XF_m^\top)\beta-(F_n X-XF_m^\top)\beta^2)P_{\alpha,\beta}^{r-2}F_m^\top.
\end{align*}
Continuing this iteration and letting $r=m+n-1$, we get after $r-1$ steps
\begin{equation}
XP_{\alpha,\beta}^{m+n-1}=\sum_{k=0}^{r-2}\left(\sum_{i=0}^{k+1}(-1)^{k+i}\binom{k+1}{i}\overline{\alpha}^{k-i+1}
M_k\beta^i\right)P_{\alpha,\beta}^{m+n-k-2}+R,
\label{3.7}
\end{equation}
where $R$ contains the terms containing factors $F_n^\ell$ and $(F_m^\top)^j$ of total degree 
$\ell+j\ge m+n-1$. Thus, either $\ell\ge n$ in which case $F_n^\ell=0$, or $j\ge m$ in which case 
$(F_m^\top)^j=0$. Hence all terms in $R$ are zero matrices (i.e., $R=0$), and formula \eqref{3.4}
follows immediately from \eqref{3.7}. Since \eqref{2.12u} is equivalent to \eqref{3.5}, each iteration
of \eqref{2.12u} (and in particular, the formula \eqref{3.4}) is equivalent to \eqref{3.5}.
Therefore, $X$ of the form \eqref{3.4} is a solution to \eqref{3.5}. The uniqueness of a solution is evident.
\end{proof}
In Theorems \ref{T:3.2a}, \ref{T:3.8} and \ref{T:3.6}, all restrictions on $B$ will be removed. 
The next theorem settles the case where $A=\mathcal J_{\balpha}$ is of the form \eqref{2.20} with
all diagonal entries be in the same conjugacy class. To formulate the theorem, let us note that 
for the matrices
\begin{equation}
A=\left[\begin{array}{ccccc} \alpha_1&0&\ldots&&0\\ 1
&\alpha_2&0&&\\ 0&1&\ddots&\ddots&\vdots\\
\vdots&\ddots&\ddots&\ddots&0\\ 0&\ldots &0&1&\alpha_n
\end{array}\right],\;
A^\prime=\left[\begin{array}{ccccc} \overline{\alpha}_1&0&\ldots&&0\\ -1
&\overline{\alpha}_2&0&&\\ 0&-1&\ddots&\ddots&\vdots\\
\vdots&\ddots&\ddots&\ddots&0\\ 0&\ldots &0&-1&\overline{\alpha}_n\end{array}\right],
\label{3.4a}
\end{equation}
based on $\alpha_1,\ldots,\alpha_n$ from the same conjugacy class, we have
\begin{equation}
A+A^\prime=(\alpha_1+\overline{\alpha}_1)I_n\quad\mbox{and}\quad
A^\prime A=|\alpha_1|^2 I_n-\widetilde{A},
\label{3.4b}
\end{equation}
where $\widetilde{A}=\left[(\alpha_i-\overline{\alpha}_{i+1})\delta_{i,j+1}+\delta_{i,j+2}\right]_{i,j=1}^n$;
more explicitly:
\begin{equation}
\widetilde{A}=\left[\begin{array}{cccccc} 0 &0&\ldots&&\ldots&0\\
\alpha_1-\overline{\alpha}_2&0&&&&\\ 1 &\alpha_2-\overline{\alpha}_3&\ddots&\ddots&\ddots&\vdots\\
0&1& \ddots&\ddots&\ddots&\vdots \\
\vdots&\ddots&\ddots&\ddots&0&0\\ 0&\ldots &0&1&\alpha_{n-1}-\overline{\alpha}_n& 
0\end{array}\right].
\label{3.4c}
\end{equation}
Relations \eqref{3.4b} follow from characterization \eqref{2.1}, due to which 
$|\alpha_i|^2=|\alpha_1|^2$ and ${\rm Re} \alpha_i={\rm Re} \alpha_1$ for $i=2,\ldots,n$.
\begin{theorem}
Let $\alpha_1,\ldots,\alpha_n$ be the elements from the same conjugacy class $V\subset\bH$,
let the matrices $A$, $A^\prime$ and $\widetilde{A}$ be defined as in \eqref{3.4a}, \eqref{3.4c},
and let $B\in\bH^{m\times m}$ be such that   
$V\cap\sigma_{\bf r}(B)=\emptyset$. Then the equation \eqref{1.10} has a unique solution
\begin{equation}
X=\sum_{k=0}^{n-1}\widetilde{A}^{k}\left(A^\prime C-CB\right)(\cX_V(B))^{-k-1},
\label{3.4e}
\end{equation}
where $\cX_V$ is the characteristic polynomial of the conjugacy class $V$.
\label{T:3.2a}
\end{theorem}
\begin{proof} We multiply equation \eqref{1.10} by $A^\prime$ on the left and by $B$ on the right 
$$
A^\prime AX-A^\prime XB=A^\prime C,\qquad AXB-XB^2=CB,
$$
and then subtract the second equality from the first. Due to \eqref{3.4b}, we get
$$
|\alpha_1|^2 X-\widetilde{A}X-(\alpha_1+\overline{\alpha}_1)XB+XB^2=A^\prime C-CB,
$$
which can be written in terms of the characteristic polynomial \eqref{2.2} as 
\begin{equation}
X\cX_V(B)=A^\prime C-CB+\widetilde{A}X.
\label{3.4d}
\end{equation}
We now iterate \eqref{3.4d} as in the proof of Theorem \ref{T:3.2}:
\begin{align*}
X(\cX_V(B))^{n}=&(A^\prime C-CB)(\cX_V(B))^{n-1}+\widetilde{A}X(\cX_V(B))^{n-1}\\
=&(A^\prime C-CB)(\cX_V(B))^{n-1}+\widetilde{A}(A^\prime C-CB)(\cX_V(B))^{n-2}\\
&+\widetilde{A}^2X(\cX_V(B))^{n-2}=\ldots \\
=&\sum_{k=0}^{n-1}\widetilde{A}^{k}\left(A^\prime C-CB\right)(\cX_V(B))^{n-k-1}+\widetilde{A}^{n}X.
\end{align*}
The last equality implies \eqref{3.4e} since $\widetilde A^n=0$ (see \eqref{3.4c}) and $\cX_V(B)$
is invertible, since $V\cap\sigma_{\bf r}(B)=\emptyset$.
\end{proof}
We next remove the assumption that all diagonal entries in $A$ are in the same conjugacy class.
In this case, we get explicit formulas for the rows of a unique solution $X$.  
In what follows, we will use the noncommutative product notation
\begin{equation}
\prod_{i=1}^{\substack{\curvearrowright \\ k}}\gamma_i:=\gamma_1\gamma_2\cdots\gamma_k,
\quad\mbox{and}\quad \prod_{i=1}^{\substack{\curvearrowleft \\ k}}\gamma_i:=\gamma_k\cdots
\gamma_2\gamma_1
\label{feb10}
\end{equation}

\begin{theorem}
Let $A=\mathcal J_{\balpha}\in\bH^{n\times n}$ be defined as in \eqref{2.20}, let $B\in\bH^{m\times m}$ 
be such that $\sigma_{\bf r}(B)\cap [\alpha_k]=\emptyset$ for $k=1,\ldots,n$, and let $\widetilde{C}_j$ 
denote the $j$-th row of the matrix $C\in\bH^{n\times m}$. Then the equation \eqref{1.10} has a unique 
solution
\begin{equation}
X={\rm Col}_{1\le k\le n}\widetilde{X}_k,
\label{3.8}
\end{equation}
with the rows given by
\begin{equation}
\widetilde{X}_k=-\sum_{j=1}^k\left(\bp_{\overline{\alpha}_k}\bp_{\overline{\alpha}_{k-1}}\cdots\bp_{\overline{\alpha}_j}
\widetilde{C}_j\right)^{\br}(B)
\cdot \prod_{i=j}^{\substack{\curvearrowright \\ k}}(\cX_{[\alpha_i]}(B))^{-1},
\label{2.21}
\end{equation}
for $k=1,\ldots,n$, where the polynomials $\bp_{\overline{\alpha}_i}$ and $\cX_{[\alpha_i]}$ are defined via 
formulas \eqref{2.13} and \eqref{2.12}, respectively.
\label{T:3.8}
\end{theorem}
The proof is based on the following two observations.
\begin{remark}
Let $f$ and $g$ be two quaternion matrix polynomials and let $B$ be a square matrix. Then 
\begin{equation}
(fg)^{\br}(B)=\left(f\cdot g^{\br}(B)\right)^{\br}(B).
\label{2.31}
\end{equation}
\label{R:3.4}
\end{remark}
Indeed, since evaluations \eqref{2.11} are linear, we have
\begin{align*}
\left(f\cdot g^{\br}(B)\right)^{\br}(B)&=(f\cdot\sum_j g_jB^j)^{\br}(B)\\
&=\sum_kf_k(\sum_j g_jB^j)B^k\\
&=\sum_i(\sum_{j+k=i}f_kg_j)B_i=(fg)^{\br}(B).
\end{align*}
\begin{remark}
Let $\bp_\alpha$ be given by \eqref{2.13} for a fixed $\alpha\in\bH$ and let
\begin{equation}
\left(\bp_\alpha D\right)^{\br}(B)=M
\label{2.30}
\end{equation}
for some $D,M\in\bH^{1\times m}$ and 
$B\in\bH^{m\times m}$ such that $\sigma_{\bf r}(B)\cap[\alpha]=\emptyset$. Then
\begin{equation}
D=\left(\bp_{\overline\alpha}M\right)^{\br}(B)\cdot (\cX_{[\alpha]}(B))^{-1}.
\label{2.32}
\end{equation}
\label{R:3.5} 
\end{remark}
Indeed, applying \eqref{2.31} to $f=\bp_{\overline{\alpha}}$ and $g=\bp_\alpha D$ and taking into account 
equalities
$\bp_{\overline{\alpha}}\bp_{\alpha}=\cX_{[\alpha]}$ and \eqref{2.30}, we get
$$
(\cX_{[\alpha]}D)^{\br}(B)=\left(\bp_{\overline{\alpha}} \cdot (\bp_\alpha D)^{\br}(B)\right)^{\br}(B)=
\left(\bp_{\overline{\alpha}}M\right)^{\br}(B).
$$
Since $\cX_{[\alpha]}$ is a polynomial with real coefficients, we have
$$
(\cX_{[\alpha]}D)^{\br}(B)=(D\cX_{[\alpha]})^{\br}(B)=D\cX_{[\alpha]}(B),
$$
and \eqref{2.32} follows from the two latter equalities, once we recall that the 
matrix $\cX_{[\alpha_1]}(B)$ is invertible (as $[\alpha]\cap \sigma_{\bf r}(B)=\emptyset$).
\medskip

\begin{proof}[Proof of Theorem \ref{T:3.8}]
Equating the corresponding rows in the equation \eqref{1.10} we get, due to 
\eqref{2.20} and \eqref{3.8}, the system
$$
\alpha_1\widetilde{X}_1-\widetilde{X}_1B=\widetilde{C}_1,\quad 
\alpha_k\widetilde{X}_k-\widetilde{X}_kB=\widetilde{C}_k-\widetilde{X}_{k-1} \; \; (k=2,\ldots,n),
$$
which is equivalent to \eqref{1.10}. The latter equalities can be written in terms of right evaluations as 
follows:
\begin{equation}
\left(\bp_{\alpha_1}\widetilde{X}_1\right)^{\br}(B)=-\widetilde{C}_1,\quad 
\left(\bp_{\alpha_k}\widetilde{X}_k\right)^{\br}(B)=\widetilde{X}_{k-1}-\widetilde{C}_k \; \; 
(k=2,\ldots,n).
\label{2.22}
\end{equation}
Making use of Remark \ref{R:3.5}, we solve the leftmost equation in \eqref{2.22}:
$$
\widetilde{X}_1=-\left(\bp_{\overline{\alpha}_1}\widetilde{C}_1\right)^{\br}(B)\cdot(\cX_{[\alpha_1]}(B))^{-1},
$$ 
confirming formula \eqref{2.21} for $\widetilde{X}_1$. Similarly, we solve the $k$-th 
equation  in \eqref{2.4} for $\widetilde{X}_k$:
\begin{equation}
\widetilde{X}_k=\left(\bp_{\overline{\alpha}_k}(\widetilde{X}_{k-1}-\widetilde{C}_k)
\right)^{\br}(B)\cdot(\cX_{[\alpha_k]}(B))^{-1}.
\label{2.23}  
\end{equation}
Assuming that the formula \eqref{2.21} holds for $\widetilde{X}_{k-1}$, we plug it into \eqref{2.23}:
\begin{align*}
\widetilde{X}_k=&-\left(\bp_{\overline{\alpha}_k}\widetilde{C}_k\right)^{\br}(B)\cdot(\cX_{[\alpha_k]}(B))^{-1}\\
&-\left(\bp_{\overline{\alpha}_k}
\sum_{j=1}^{k-1}\left(\bp_{\overline{\alpha}_{k-1}}\cdots\bp_{\overline{\alpha}_j}
\widetilde{C}_j\right)^{\br}(B)\cdot 
\prod_{i=j}^{\substack{\curvearrowright \\ k-1}}(\cX_{[\alpha_i]}(B))^{-1}
\right)^{\br}(B)\\
&\quad \cdot(\cX_{[\alpha_k]}(B))^{-1}\\
=&-\left(\bp_{\overline{\alpha}_k}\widetilde{C}_k\right)^{\br}(B)\cdot(\cX_{[\alpha_k]}(B))^{-1}\\
&-\sum_{j=1}^{k-1} \left(\bp_{\overline{\alpha}_k}\bp_{\overline{\alpha}_{k-1}}\cdots\bp_{\overline{\alpha}_j}
\widetilde{C}_j\right)^{\br}(B)\cdot 
\prod_{i=j}^{\substack{\curvearrowright \\ k}}(\cX_{[\alpha_i]}(B))^{-1}\\
=&-\sum_{j=1}^k\left(\bp_{\overline{\alpha}_k}\cdots\bp_{\overline{\alpha}_j}
\widetilde{C}_j\right)^{\br}(B) \cdot \prod_{i=j}^{\substack{\curvearrowright \\ 
k}}(\cX_{[\alpha_i]}(B))^{-1}
\end{align*}
and the induction argument completes the proof of formulas \eqref{2.21}. Note that the second equality
in the last calculation followed by Remark \ref{R:3.4} applied to polynomials 
$f=\bp_{\alpha_k}$ and 
$g=\bp_{\overline{\alpha}_{k-1}}\cdots\bp_{\overline{\alpha}_j}\widetilde{C}_j$ and since the 
characteristic polynomial $\cX_{[\alpha_i]}$ is in $\R[z]$ for $i=1,\ldots,n$.
\end{proof}
\begin{remark}
If the matrix $A=\mathcal J_{\balpha}$ is based on the elements $\alpha_1,\ldots\alpha_n$ from the same conjugacy 
class $V$, then formulas \eqref{2.21} simplify to
$$
\widetilde{X}_k=-\sum_{j=1}^k\left(\bp_{\overline{\alpha}_k}\bp_{\overline{\alpha}_{k-1}}\cdots
\bp_{\overline{\alpha}_j}\widetilde{C}_j\right)^{\br}(B)
\cdot (\cX_{V})(B))^{j-k-1}.
$$
\label{R:3.6}
\end{remark}
Making use of Remark \ref{R:3.5}, one can get the formula for the unique solution 
of the equation \eqref{1.10} in the case where $A$ is lower triangular.
\begin{theorem}
Let $A=\left[\alpha_{i,j}\right]_{i,j=1}^n$ be a lower triangular matrix ($\alpha_{i,j}=0$ for $i<j$)
and let $B\in\bH^{m\times m}$ be such that $\sigma_{\bf r}(B)\cap [\alpha_{k,k}]=\emptyset$ for 
$k=1,\ldots,n$. Then the equation \eqref{1.10} has a unique solution $X$ of the form \eqref{3.8}
with the rows defined recursively by 
\begin{align}
\widetilde{X}_1=&-\left(\bp_{\overline{\alpha}_{1,1}}\widetilde{C}_1\right)^{\br}(B)
\cdot(\cX_{[\alpha_{1,1}]}(B))^{-1},\notag\\
\widetilde{X}_{k}=&-\left(\bp_{\overline{\alpha}_{k,k}}\widetilde{C}_k\right)^{\br}(B)
\cdot(\cX_{[\alpha_{k,k}]}(B))^{-1}\label{3.16
}\\
&+\sum_{j=1}^{k-1}\left(\bp_{\overline{\alpha}_{k,k}}\alpha_{k,j}\widetilde{X}_j\right)^{\br}(B)
\cdot(\cX_{[\alpha_{k,k}]}(B))^{-1}\quad\mbox{for}\quad k=2,\ldots,n.\notag
\end{align}
\label{T:3.6}
\end{theorem}
\begin{proof} Equating the corresponding rows in the equation \eqref{1.10} we get, due to the lower 
triangular structure of 
$A$,
the equations which can be written in terms of right evaluations as 
$$\left(\bp_{\alpha_1}\widetilde{X}_1\right)^{\br}(B)=-\widetilde{C}_1,\quad
\left(\bp_{\alpha_k}\widetilde{X}_k\right)^{\br}(B)=-\widetilde{C}_k+\sum_{j=1}^{k-1}\alpha_{k,j}
\widetilde{X}_{j}
$$
for $k=2,\ldots,n$. The rest follows by Remark \ref{R:3.5}.
\end{proof}
To keep the presentation symmetric we conclude with the ``column" version of the last theorem.
\begin{theorem}
Let $B=\left[\beta_{i,j}\right]_{i,j=1}^m$ be an upper triangular matrix ($\beta_{i,j}=0$ for $i>j$)
and let $A\in\bH^{n\times n}$ be such that $\sigma_{\bf r}(A)\cap [\beta_{k,k}]=\emptyset$ for
$k=1,\ldots,m$ and let $C=\begin{bmatrix}C_1 & \ldots & C_m\end{bmatrix}$. 
Then the equation \eqref{1.10} has a unique solution $X=\begin{bmatrix}
X_1 & \ldots & X_m\end{bmatrix}$ with the columns defined recursively by
\begin{align}
X_1=&(\cX_{[\beta_{1,1}]}(A))^{-1}\cdot
\left(C_1\bp_{\overline{\beta}_{1,1}}\right)^{\bl}(A),\notag\\
X_{k}=&(\cX_{[\beta_{k,k}]}(A))^{-1}\cdot
\left(C_k{\overline{\beta}_{k,k}}\right)^{\bl}(A)\label{3.17}\\
&+(\cX_{[\beta_{k,k}]}(A))^{-1}\sum_{j=1}^{k-1}
\left(X_j\beta_{j,k}\bp_{\overline{\beta}_{k,k}}\right)^{\bl}(A)
\quad\mbox{for}\quad k=2,\ldots,m.\notag   
\end{align}   
\label{T:3.7} 
\end{theorem}
\begin{proof} Let us observe that the left and right evaluations \eqref{2.11} are related as follows: for any 
$\alpha\in\bH$, $D\in\bH^{1\times m}$ and $B\in\bH^{m\times m}$, 
\begin{equation}
\left(\bp_{\alpha}C\right)^{\br}(B)=CB-\alpha 
C=\left(B^*C^*-C^*\overline{\alpha}\right)^*=\left[\left(C^*\bp_{\overline{\alpha}}\right)^{\bl}(B^*)\right]^*.
\label{3.18}  
\end{equation}
Taking adjoints \eqref{1.10} we get the equation 
\begin{equation}
B^*X^*-X^*A^*=-C^*,
\label{3.18a}
\end{equation}
and since the matrix $B^*$ is lower triangular, we can apply Theorem \ref{T:3.6} (with $A$, $B$, $C$ replaced 
by $B^*$, $A^*$ and $-C^*$, respectively) to get recursive formulas for the rows of the matrix $X^*$.
Taking adjoints in these formulas and making use of relations \eqref{3.18}, we get \eqref{3.17}.
\end{proof}
To conclude, we remark that in case $B={\mathcal J}_{\bbeta}^\top$ where $\mathcal J_{\bbeta}^\top$ is of the 
form \eqref{2.20}, the recursion  \eqref{3.17} can be solved to produce formulas
$$
X_k=\sum_{j=1}^k \prod_{i=j}^{\substack{\curvearrowleft \\ k}}(\cX_{[\beta_i]}(A))^{-1}
\cdot \left(C_j\bp_{\overline{\beta}_j}\bp_{\overline{\beta}_{j+1}}\cdots\bp_{\overline{\beta}_k}\right)^{\bl}(A)
\quad (1\le 1\le m).
$$

\section{Singular case: the solvability criterion} 
\setcounter{equation}{0}

In this section, we consider the Sylvester equation \eqref{1.10} in the case where 
$\sigma_{\bf r}(A)\cap\sigma_{\bf r}(B)\neq \emptyset$. At certain level, this case also
can be handled by making use of complex representation of quaternion matrices: one can 
pass from \eqref{1.10} to the equivalent complex Sylvester equation \eqref{3.1}, then
use any avaliable method (e.g., Sylvester's tensor-product approach recalled at the beginning of the paper) 
to describe all its solutions, and then claim that the formula
$X=\psi(Y)$ describes all solutions to the Sylvester equation
\eqref{1.10} when $Y$ runs through the set of all solutions to the equation \eqref{3.1}. 
This approach is not quite satisfactory, partly because the map $\psi$ \eqref{2.7} is not injective.
Presumably, some quantative results still can be obtained on this way: 
to define the number of independent conditions (scalar equalities) 
which are necessary and sufficient for the equation \eqref{1.10} to have a solution along with the 
number of independent free parameters needed to parametrize the solution set.
In this section we will obtain more definitive results of this sort (explicit solvability conditions
and parametrization of all solutions) in the case where the matrices 
\begin{equation}
A={\mathcal J}_{\balpha}=
\left[\alpha_i\delta_{i,j}+\delta_{i,j+1}\right]_{i,j=1}^n,\qquad
B={\mathcal J}_{\bbeta}^\top=\left[\beta_i\delta_{i,j}+\delta_{i+1,j}\right]_{i,j=1}^m
\label{4.1}  
\end{equation}
are based on spherical chains $\balpha=(\alpha_1,\ldots,\alpha_n)$ and 
$\bbeta=(\beta_1,\ldots,\beta_m)$. We assume without loss of generality 
that $n\ge m$ and start with the scalar equation  
\begin{equation}
\alpha x-x\beta =c,\quad \mbox{where}\quad \alpha\sim\beta.
\label{4.2}
\end{equation}
Let us recal that a unit element $I\in\bH$ is purely imaginary if and only if $I^2=-1$. Therefore, the 
characterization \eqref{2.1} can be reformulated as follows: {\em $\alpha\sim\beta$ if and only if
$\alpha$ and $\beta$ can be written as}
\begin{equation}
\alpha=x+y I,\quad \beta=x+y \widetilde{I} \qquad (x\in\R, \; y>0, \; I^2=\widetilde{I}^2=-1).
\label{4.3}
\end{equation}
Since $\bH$ is a (four-dimensional) vector space over $\R$, we may define
orthogonal complements with respect to the usual euclidean metric in $\R^4$.
For $\alpha$ and $\beta$ as in \eqref{4.3}, we define the plane (the two-dimensional
subspace of $\bH\cong\R^4$) $\Pi_{\alpha,\beta}$ via the formula
\begin{equation}
\Pi_{\alpha,\beta}=\left\{\begin{array}{lll}
span \{1,I\}=\{u+vI: \, u,v\in\R\},&\mbox{if}& \beta=\alpha,\\
\left(span \{1,I\}\right)^\perp,&\mbox{if}& \beta=\overline{\alpha},\\
span \{I+\widetilde{I}, \;  1-I\widetilde{I}\},&\mbox{if}& \beta\neq \alpha,\overline{\alpha}.
\end{array}\right.
\label{4.4}
\end{equation}
Since $\overline{\alpha}=x-y I$, it follows that
$\Pi_{\overline{\alpha},\overline{\alpha}}=\Pi_{\alpha,\alpha}$,
$\Pi_{\overline{\alpha},\alpha}=\Pi_{\alpha, \overline{\alpha}}$ and
\begin{equation}
\Pi_{\overline{\alpha},\beta}=span \{I-\widetilde{I}, \;  1+I\widetilde{I}\}\quad
\mbox{if}\quad\beta\neq \alpha,\overline{\alpha}.
\label{4.5}  
\end{equation}
\begin{lemma}[\cite{bol}]   
Given $\alpha\sim\beta$, the solution set of the homogeneous Sylvester equation $\alpha x=x\beta$ 
coincides with
$\Pi_{\alpha,\beta}$ given in \eqref{4.4}. Furthermore, the non-homogeneous equation \eqref{4.2} has a 
solution if and only if $\overline{\alpha}c=c\beta$ (equivalently, $c\in\Pi_{\overline{\alpha},\beta}$) 
in which case  the solution set  
is the affine plane
$$
(2{\rm Im}(\alpha))^{-1}c+\Pi_{\alpha,\beta}=-c(2{\rm Im}(\beta))^{-1}+\Pi_{\alpha,\beta}.
$$
\label{L:6.2}
\end{lemma}
We now proceed to the matrix equation
\begin{equation}
AX-XB=C,\quad X=[x_{ij}]=\begin{bmatrix}X_1 &\ldots & X_m\end{bmatrix},
\label{4.6}   
\end{equation}
where $A$ and $B$ of the form \eqref{4.1} are based on the spherical chains $\balpha$ and 
$\bbeta$ from the same conjugacy class $V$ and where we assume without loss of generality 
that $m\le n$. With the given matrices $A$, $B$ and $C$, we associate the matrix 
\begin{equation}
D:=CB-A^\prime C
\label{4.7}   
\end{equation}
where $A^\prime$ is given in \eqref{3.4a}. In more detail, if 
$C=\left[c_{i,j}\right]_{i=1,\ldots,n}^{j=1,\ldots,m}$ and if we let 
$c_{i,0}=c_{0,j}=0$ for all
$i,j$, then  
\begin{equation}
D=\left[d_{i,j}\right]_{i=1,\ldots,n}^{j=1,\ldots,m},\quad 
d_{i,j}=c_{i,j}\beta_j-\overline{\alpha}_ic_{i,j}+c_{i,j-1}+c_{i-1,j}.
\label{4.8}   
\end{equation}
We next introduce the elements $\Gamma_{k,j}\in\bH$ by the following double recursion
\begin{equation}
\Gamma_{k,j}=(\alpha_k-\overline{\alpha}_{k+1})^{-1}\left[d_{k+1,j}+
\Gamma_{k+1,j-2}-\Gamma_{k-1,j}-\Gamma_{k+1,j-1}(\overline{\beta}_j-\beta_{j-1})
\right]
\label{4.9}
\end{equation}
with the initial conditions 
\begin{equation}
\Gamma_{-1,j}=\Gamma_{0,j}=\Gamma_{k,0}=0\quad\mbox{for all}\quad k,j\ge 1.
\label{4.10}
\end{equation}
It is clear from \eqref{4.9} that the assumption that $\balpha$ is a spheerical chain 
(i.e., that $\alpha_k\neq \overline{\alpha}_{k+1}$) is essential. We make several further 
comments.
\begin{remark}
{\rm (1) Recursion \eqref{4.9} determines $\Gamma_{k,j}$ for all positive $k<n$ and 
$j\le m$  such that $k+j\le n$.

\smallskip
\noindent
(2) Any element $\Gamma_{k,1}$ from the first ``column" is determined by 
    the elements $\Gamma_{r,1}$ ($1\le r<k$).

\smallskip
\noindent 
(3) Any element in the $\ell$-th counter-diagonal $\mathcal D_\ell=\{\Gamma_{k,j}: \, 
k+j=\ell+1\}$ can be expressed in terms of the elements from the previous 
counter-diagonal $\mathcal D_{\ell-1}=\{\Gamma_{k,j}: \, k+j=\ell\}$ and one fixed 
element in $\mathcal D_\ell$. The latter follows from the formula \eqref{4.9}
since $\beta_j\neq \overline{\beta}_j$ ($\bbeta$ is a spherical chain).}
\label{R:4.1}
\end{remark}
\begin{lemma}
Let $\Gamma_{k,j}$ be defined as in \eqref{4.8}--\eqref{4.10} and let 
\begin{equation}
S_j=d_{1,j}+\Gamma_{1,j-1}(\beta_{j-1}-\overline{\beta}_j)+\Gamma_{1,j-2}\quad (j=1,\ldots,m).
\label{4.11}  
\end{equation}
If $S_j=0$ for $j=1,\ldots,m$, then 
\begin{equation}
\Delta_{k,j}:=\alpha_k\Gamma_{k,j}-\Gamma_{k,j}\beta_j-c_{k,j}+\Gamma_{k-1,j}-\Gamma_{k,j-1}=0
\label{4.12}
\end{equation}
for $k<n$ and $j\le m$  such that $k+j\le n$.
\label{L:6.3}
\end{lemma}
\begin{proof} 
We first observe that for $\alpha\sim\beta$,
\begin{equation}
\alpha(\alpha-\overline{\beta})^{-1}=(\alpha-\overline{\beta})^{-1}\beta,\quad
\alpha-\overline{\beta}=\beta-\overline{\alpha}.
\label{4.13}
\end{equation}
By \eqref{4.8}--\eqref{4.10},
\begin{align}
\Gamma_{k,1}&=(\alpha_k-\overline{\alpha}_{k+1})^{-1}\left[d_{k+1,1}-\Gamma_{k-1,1}\right]\notag\\
&=(\alpha_k-\overline{\alpha}_{k+1})^{-1}\left[c_{k+1,1}\beta_1-\overline{\alpha}_{k+1}c_{k+1,1}
+c_{k,1}-\Gamma_{k-1,1}\right].\label{4.14}
\end{align}
Making use of the first equality (with $\alpha=\alpha_k$ and $\beta=\alpha_{k+1}$), we have,
on account of \eqref{4.7} and \eqref{4.14},
\begin{align}
\Delta_{k,1}=&\alpha_k\Gamma_{k,1}-\Gamma_{k,1}\beta_1-c_{k,1}+\Gamma_{k-1,1}\notag\\
=&(\alpha_k-\overline{\alpha}_{k+1})^{-1}\left[
\alpha_{k+1}\left(c_{k+1,1}\beta_1-\overline{\alpha}_{k+1}c_{k+1,1}
+c_{k,1}-\Gamma_{k-1,1}\right)\right.\notag\\
&-\left.\left(c_{k+1,1}\beta_1-\overline{\alpha}_{k+1}c_{k+1,1}
+c_{k,1}-\Gamma_{k-1,1}\right)\beta_1\right]-c_{k,1}+\Gamma_{k-1,1}\notag\\
=&(\alpha_k-\overline{\alpha}_{k+1})^{-1}\left[
-c_{k+1}\cX_{[\alpha_{k+1}]}(\beta_1)+\overline{\alpha}_k c_{k_1}-c_{k_1}\beta_1\right.\notag\\
&\left. -\overline{\alpha}_k\Gamma_{k-1,1}+\Gamma_{k-1,1}\beta_1\right]\notag\\
=&(\overline{\alpha}_{k+1}-\alpha_k)^{-1}\left[d_{k,1}-c_{k-1,1}
+\overline{\alpha}_{k}\Gamma_{k-1,1}-\Gamma_{k-1,j}\beta_1\right],\label{4.15}
\end{align}
where the last equality follows since $\beta_j\sim\alpha_{k+1}$ so that
$\cX_{[\alpha_{k+1}]}(\beta_j)=0$, and due to formula \eqref{4.8} for $d_{k,1}$.
Letting $k=1$ in \eqref{4.15} and taking \eqref{4.6} into account, we get
\begin{equation}
\Delta_{1,1}=(\overline{\alpha}_{2}-\alpha_1)^{-1}d_{1,1}=(\overline{\alpha}_{2}-\alpha_1)^{-1}S_1.
\label{4.17}
\end{equation}
For $k>1$, we have
\begin{align*}
\overline{\alpha}_k\Gamma_{k-1,j}-\Gamma_{k-1,j}\beta_j&=
\Delta_{k-1,1}+(\overline{\alpha}_k-\alpha_{k-1})\Gamma_{k-1,1}+c_{k-1,1}-\Gamma_{k-2,1}\\
&=\Delta_{k-1,1}+c_{k-1,1}-d_{k,1}
\end{align*}
where the first equality follows from formula \eqref{4.12} for $\Delta_{k-1,1}$, and 
the second equality follows from \eqref{4.14} (with $k$ replaced by $k-1$). Combining the latter 
equality with \eqref{4.15} gives
$$
\Delta_{k,1}=(\overline{\alpha}_{k+1}-\alpha_k)^{-1}\Delta_{k-1,1},
$$
from which, on account of \eqref{4.17},  we recursively obtain 
\begin{equation}
\Delta_{k,1}=(\overline{\alpha}_{k+1}-\alpha_k)^{-1}(\overline{\alpha}_{k}-\alpha_{k-1})^{-1}
\cdots (\overline{\alpha}_{2}-\alpha_1)^{-1}S_1.
\label{feb7}
\end{equation}
Since $S_1=0$, it follows that 
\begin{equation}
\Delta_{k,1}:=\alpha_k\Gamma_{k,1}-\Gamma_{k,1}\beta_1-c_{k,1}+\Gamma_{k-1,1}=0 
\quad (k=1,\ldots,n-1).
\label{4.18}
\end{equation}
We now assume that $j\ge 2$. Making use of the first equality (with $\alpha=\alpha_k$ 
and $\beta=\alpha_{k+1}$), we have, on account of \eqref{4.9},
\begin{align}
&\alpha_k\Gamma_{k,j}-\Gamma_{k,j}\beta_j\label{4.19}\\
&=\alpha_k(\alpha_k-\overline{\alpha}_{k+1})^{-1}
\left[d_{k+1,j}+
\Gamma_{k+1,j-2}-\Gamma_{k-1,j}-\Gamma_{k+1,j-1}(\overline{\beta}_j-\beta_{j-1})\right]\notag\\
&\quad - (\alpha_k-\overline{\alpha}_{k+1})^{-1}\left[d_{k+1,j}+
\Gamma_{k+1,j-2}-\Gamma_{k-1,j}-\Gamma_{k+1,j-1}(\overline{\beta}_j-\beta_{j-1})
\right]\beta_j\notag\\
&=(\alpha_k-\overline{\alpha}_{k+1})^{-1}\left[
\alpha_{k+1}d_{k+1,j}-d_{k+1,j}\beta_j+
\alpha_{k+1}\Gamma_{k+1,j-2}-\Gamma_{k+1,j-2}\beta_j\right.\notag\\
&\quad \left.-\alpha_{k+1}\Gamma_{k-1,j}+\Gamma_{k-1,j}\beta_j-(\alpha_{k+1}\Gamma_{k+1,j-1}-
\Gamma_{k+1,j-1}\beta_{j-1})(\overline{\beta}_j-\beta_{j-1})\right].\notag
\end{align}
Observe that in view of \eqref{4.8} and \eqref{2.12},
\begin{align*}
\alpha_{k+1}d_{k+1,j}-d_{k+1,j}\beta_j=&\alpha_{k+1}
(c_{k+1,j}\beta_j-\overline{\alpha}_{k+1}c_{k+1,j}+c_{k+1,j-1}+c_{k,j})\\
&-(c_{k+1,j}\beta_j-\overline{\alpha}_{k+1}c_{k+1,j}+c_{k+1,j-1}+c_{k,j})\beta_j\\
=&c_{k+1,j}\cX_{[\alpha_{k+1}]}(\beta_j)+\alpha_{k+1}(c_{k+1,j-1}+c_{k,j})\\
&-(c_{k+1,j-1}+c_{k,j})\beta_j\\
=&\alpha_{k+1}(c_{k+1,j-1}+c_{k,j})-(c_{k+1,j-1}+c_{k,j})\beta_j,
\end{align*}
where the last equality follows since $\beta_j\sim\alpha_{k+1}$ so that
$\cX_{[\alpha_{k+1}]}(\beta_j)=0$. We next observe equalities
\begin{align*}
\alpha_{k+1}c_{k+1,j-1}-c_{k+1,j-1}\overline{\beta}_{j-1}&=
c_{k+1,j-1}\beta_{j-1}-\overline{\alpha}_{k+1}c_{k+1,j-1}\notag\\
&=d_{k+1,j-1}-c_{k+1,j-2}-c_{k,j-1}\\
\overline{\alpha}_{k}c_{k,j}-c_{k,j}\beta_{j}&=c_{k,j-1}+c_{k-1,j}-d_{k,j},
\end{align*}
which follow from \eqref{4.8} and the fact that the elements from the same
conjugacy class have the same real part. Combining the three last equalities gives
\begin{align}
\alpha_{k+1}d_{k+1,j}-d_{k+1,j}\beta_j=&d_{k+1,j-1}-d_{k,j}
-c_{k+1,j-2}+c_{k-1,j}\notag\\
&+(\alpha_k-\overline{\alpha}_{k+1})c_{k,j}+c_{k+1,j-1}(\overline{\beta}_j-\beta_{j-1}).
\notag
\end{align}
We now substitute the latter equality into \eqref{4.19} and then \eqref{4.19} into 
the definition \eqref{4.12} of $\Delta_{k,j}$ to conclude
\begin{align}
\Delta_{k,j}=&\alpha_k\Gamma_{k,j}-\Gamma_{k,j}\beta_j-c_{k,j}+\Gamma_{k-1,j}-\Gamma_{k,j-1}\notag\\
=&(\alpha_k-\overline{\alpha}_{k+1})^{-1}\left[
d_{k+1,j-1}-d_{k,j}-c_{k+1,j-2}+c_{k-1,j}\right.\notag\\
&+(\alpha_k-\overline{\alpha}_{k+1})c_{k,j}+c_{k+1,j-1}(\overline{\beta}_j-\beta_{j-1})
+\alpha_{k+1}\Gamma_{k+1,j-2}\notag\\
& -\Gamma_{k+1,j-2}\beta_j
-\alpha_{k+1}\Gamma_{k-1,j}+\Gamma_{k-1,j}\beta_j\notag\\
&\left. -(\alpha_{k+1}\Gamma_{k+1,j-1}-
\Gamma_{k+1,j-1}\beta_{j-1})(\overline{\beta}_j-\beta_{j-1})\right]\notag\\
&-c_{k,j}+\Gamma_{k-1,j}-\Gamma_{k,j-1}\notag\\
=&(\alpha_k-\overline{\alpha}_{k+1})^{-1}\left[
d_{k+1,j-1}-d_{k,j}-c_{k+1,j-2}+c_{k-1,j}\right.\notag\\
&-\overline{\alpha}_{k}\Gamma_{k-1,j}+\Gamma_{k-1,j}\beta_j 
+\alpha_{k+1}\Gamma_{k+1,j-2}-\Gamma_{k+1,j-2}\beta_j\label{4.20}\\
&\left. -(\alpha_{k+1}\Gamma_{k+1,j-1}-
\Gamma_{k+1,j-1}\beta_{j-1}-c_{k+1,j-1})(\overline{\beta}_j-\beta_{j-1})\right]-\Gamma_{k,j-1}.
\notag
\end{align}
Letting $j=2$ in the latter equality gives
\begin{align}
\Delta_{k,2}=&(\alpha_k-\overline{\alpha}_{k+1})^{-1}\left[
d_{k+1,1}-d_{k,2}+c_{k-1,2}
-\overline{\alpha}_{k}\Gamma_{k-1,2}+\Gamma_{k-1,2}\beta_2\right.\notag\\
&\left. -(\alpha_{k+1}\Gamma_{k+1,1}-
\Gamma_{k+1,1}\beta_{1}-c_{k+1,1})(\overline{\beta}_2-\beta_{1})\right]-\Gamma_{k,1}.
\label{4.21}
\end{align}  
Taking into account the first equality in \eqref{4.14} and equality 
$$
\alpha_{k+1}\Gamma_{k+1,1}-\Gamma_{k+1,1}\beta_{1}-c_{k+1,1}=-\Gamma_{k,1}
$$
which is a consequence of \eqref{4.18} (with $k+1$ instead of $k$), we simplify
\eqref{4.21} as follows: 
\begin{align}
\Delta_{k,2}=&(\overline{\alpha}_{k+1}-\alpha_k)^{-1}\left[
d_{k,2}-c_{k-1,2}+\overline{\alpha}_{k}\Gamma_{k-1,2}-\Gamma_{k-1,2}\beta_2\right.\notag\\
&\qquad\left. +\Gamma_{k,1}(\beta_1-\overline{\beta}_2)- \Gamma_{k-1,1}\right].
\label{4.22}
\end{align}
Letting $k=1$ in \eqref{4.22} gives, on account of \eqref{4.11},
\begin{equation}
\Delta_{1,2}=(\overline{\alpha}_{2}-\alpha_1)^{-1}\left[d_{1,2}
 +\Gamma_{1,1}(\beta_1-\overline{\beta}_2)\right]=
(\overline{\alpha}_{2}-\alpha_1)^{-1}S_2.
\label{4.23}
\end{equation}
On the other hand, if $k\ge 2$, then 
$$
(\alpha_{k-1}-\overline{\alpha}_{k})\Gamma_{k-1,2}=
d_{k,2}-\Gamma_{k-2,2}+\Gamma_{k,1}(\beta_1-\overline{\beta}_2),
$$
by formula \eqref{4.9} (for $k$ replaced by $k-1$) and therefore, 
by the definition \eqref{4.12} for $\Delta_{k-1,2}$,
\begin{align*}
\Delta_{k-1,2}=&(\alpha_{k-1}-\overline{\alpha}_{k})\Gamma_{k-1,2}+
\overline{\alpha}_{k}\Gamma_{k-1,2}-\Gamma_{k-1,2}\beta_2 
-c_{k-1,2}\\
&+\Gamma_{k-2,2}-\Gamma_{k-1,1}\\
=&d_{k,2}+\Gamma_{k,1}(\beta_1-\overline{\beta}_2)
+\overline{\alpha}_{k}\Gamma_{k-1,2}-\Gamma_{k-1,2}\beta_2
-c_{k-1,2}-\Gamma_{k-1,1},
\end{align*}
which, being substituted into \eqref{4.22}, leads us to
\begin{equation}
\Delta_{k,2}=(\overline{\alpha}_{k+1}-\alpha_k)^{-1}\Delta_{k-1,2}\quad\mbox{for}\quad
k=2,\ldots,n-2.
\label{4.24}
\end{equation}
We now recursively obtain from \eqref{4.24} and \eqref{4.23} that 
$$
\Delta_{k,2}=(\overline{\alpha}_{k+1}-\alpha_k)^{-1}(\overline{\alpha}_{k}-\alpha_{k-1})^{-1}
\cdots (\overline{\alpha}_{2}-\alpha_1)^{-1}S_2,
$$
and since $S_2=0$, it follows that $\Delta_{k,2}=0$ for $k=1,\ldots,n-2$.

\smallskip

The rest will be verified by induction in $j$. Let us assume that 
\begin{equation}
\Delta_{k,\ell}=0\quad\mbox{for all  $k<n$ and $\ell<j$  such that $k+\ell\le n$.} 
\label{4.25}
\end{equation}
In particular, 
\begin{align*}
\Delta_{k+1,j-1}=&\alpha_{k+1}\Gamma_{k+1,j-1}-\Gamma_{k+1,j-1}\beta_{j-1}\notag\\
&-c_{k+1,j-1}+\Gamma_{k,j-1}-\Gamma_{k+1,j-2}=0,
\end{align*}
so that 
\begin{equation}
\alpha_{k+1}\Gamma_{k+1,j-1}-\Gamma_{k+1,j-1}\beta_{j-1}-c_{k+1,j-1}
=\Gamma_{k+1,j-2}-\Gamma_{k,j-1}.
\label{feb6}
\end{equation}
Substituting the latter equality into \eqref{4.20} gives
\begin{align}   
\Delta_{k,j}=&(\alpha_k-\overline{\alpha}_{k+1})^{-1}\left[
d_{k+1,j-1}-d_{k,j}-c_{k+1,j-2}+c_{k-1,j}\right.\notag\\
&-\overline{\alpha}_{k}\Gamma_{k-1,j}+\Gamma_{k-1,j}\beta_j
+\alpha_{k+1}\Gamma_{k+1,j-2}-\Gamma_{k+1,j-2}\beta_j\notag\\
&\left. -(\Gamma_{k+1,j-2}-\Gamma_{k,j-1})
(\overline{\beta}_j-\beta_{j-1})\right]-\Gamma_{k,j-1}\notag\\
=&(\alpha_k-\overline{\alpha}_{k+1})^{-1}\left[
d_{k+1,j-1}-d_{k,j}-c_{k+1,j-2}+c_{k-1,j}\right.\notag\\
&-\overline{\alpha}_{k}\Gamma_{k-1,j}+\Gamma_{k-1,j}\beta_j
+\alpha_{k+1}\Gamma_{k+1,j-2}-\Gamma_{k+1,j-2}\overline{\beta}_{j-1}\notag\\        
&\left. +\Gamma_{k,j-1}
(\overline{\beta}_j-\beta_{j-1})-(\alpha_k-\overline{\alpha}_{k+1}) 
\Gamma_{k,j-1}\right].\label{4.26}
\end{align}
By the definition \eqref{4.12} of $\Delta_{k+1,j-2}$, 
\begin{align*}
\alpha_{k+1}\Gamma_{k+1,j-2}-\Gamma_{k+1,j-2}\overline{\beta}_{j-1}=&
\Delta_{k+1,j-2}+c_{k+1,j-2}-\Gamma_{k,j-2}\\
&+\Gamma_{k+1,j-3}+\Gamma_{k+1,j-2}(\beta_{j-2}-\overline{\beta}_{j-1}),
\end{align*}
and since, by formula \eqref{4.9} (with $j-1$ instead of $j$),
\begin{align*}
\Gamma_{k+1,j-2}(\beta_{j-2}-\overline{\beta}_{j-1})=&
(\alpha_k-\overline{\alpha}_{k+1})\Gamma_{k,j-1}-d_{k+1,j-1}\\
&-\Gamma_{k+1,j-3}+\Gamma_{k-1,j-1},
\end{align*}
combining the two latter equalities with the assumption \eqref{4.25} gives
\begin{align}
&\alpha_{k+1}\Gamma_{k+1,j-2}-\Gamma_{k+1,j-2}\overline{\beta}_{j-1}+d_{k+1,j-1}-c_{k+1,j-2}
-(\alpha_k-\overline{\alpha}_{k+1})\Gamma_{k,j-1}\notag\\
&=\Gamma_{k-1,j-1}-\Gamma_{k,j-2}.\label{feb6b}
\end{align}
Substituting the latter equality into \eqref{4.26} gives
\begin{align}
\Delta_{k,j}=&(\overline{\alpha}_{k+1}-\alpha_k)^{-1}\left[
d_{k,j}-c_{k-1,j}+\overline{\alpha}_{k}\Gamma_{k-1,j}-\Gamma_{k-1,j}\beta_j\right.
\notag\\
&\left. +\Gamma_{k,j-1}
(\beta_{j-1}-\overline{\beta}_j)-\Gamma_{k-1,j-1}+\Gamma_{k,j-2}\right].\label{4.27}
\end{align}
Letting $k=1$ in \eqref{4.27} we get, on account of \eqref{4.11},
\begin{align}
\Delta_{1,j}=&(\overline{\alpha}_{2}-\alpha_1)^{-1}\left[d_{1,j}+\Gamma_{1,j-1}
(\beta_{j-1}-\overline{\beta}_j)+\Gamma_{1,j-2}\right]\notag\\
=&(\overline{\alpha}_{2}-\alpha_1)^{-1}
S_j.\label{4.28}
\end{align}
If $k\ge 2$, then by the definition \eqref{4.12} of $\Delta_{k-1,j}$ and 
by formula \eqref{4.9} (with $k-1$ instead of $k$), we have 
\begin{align}
\Delta_{k-1,j}=&(\alpha_{k-1}-\overline{\alpha}_k)\Gamma_{k-1,j}+
\overline{\alpha}_k\Gamma_{k-1,j}-\Gamma_{k-1,j}{\beta}_{j}\notag\\
&-c_{k-1,j}+\Gamma_{k-2,j}-\Gamma_{k-1,j-1}\notag\\
=&d_{k,j}+\Gamma_{k,j-2}-\Gamma_{k,j-1}(\overline{\beta}_{j}-\beta_{j-1})
+\overline{\alpha}_k\Gamma_{k-1,j}-\Gamma_{k-1,j}{\beta}_{j}\notag\\
&-c_{k-1,j}-\Gamma_{k-1,j-1}\label{feb6d}
\end{align}
which together with \eqref{4.28} implies
\begin{equation}
\Delta_{k,j}=(\overline{\alpha}_{k+1}-\alpha_k)^{-1}\Delta_{k-1,j}\quad\mbox{for}\quad
k=2,\ldots,n-j.
\label{4.29}
\end{equation}
We now recursively obtain from \eqref{4.29} and \eqref{4.28} that
$$
\Delta_{k,j}=(\overline{\alpha}_{k+1}-\alpha_k)^{-1}(\overline{\alpha}_{k}-\alpha_{k-1})^{-1}
\cdots (\overline{\alpha}_{2}-\alpha_1)^{-1}S_j,
$$
and since $S_j=0$, it follows that $\Delta_{k,j}=0$ for $k=1,\ldots,n-j$. The induction 
argument completes the proof of the lemma.
\end{proof}
\begin{remark}
The elements $\Delta_{k,j}$ defined as in \eqref{4.12} satisfy relations
\begin{equation}
\Delta_{k,j}
=(\overline{\alpha}_{k+1}-\alpha_k)^{-1}\left[
\Delta_{k+1,j-1}(\overline{\beta}_j-\beta_{j-1})
-\Delta_{k+1,j-2}+\Delta_{k-1,j}\right]
\label{4.6fb}
\end{equation}
for $k\ge 2$ and $j\ge 3$.
\label{R:4.5}
\end{remark}
\begin{proof}
Although we do not assume equalities \eqref{4.11} to be in force, 
formula \eqref{4.20} still holds true. Without assumptions 
\eqref{4.25}, equality \eqref{feb6} gets the extra term $\Delta_{k+1,j-1}$
on the right, so that formula \eqref{4.26} takes the form
\begin{align}
\Delta_{k,j}=&(\alpha_k-\overline{\alpha}_{k+1})^{-1}\left[
d_{k+1,j-1}-d_{k,j}-c_{k+1,j-2}+c_{k-1,j}\right.\notag\\
&-\overline{\alpha}_{k}\Gamma_{k-1,j}+\Gamma_{k-1,j}\beta_j
+\alpha_{k+1}\Gamma_{k+1,j-2}-\Gamma_{k+1,j-2}\overline{\beta}_{j-1}\notag\\
&\left. -(\Delta_{k+1,j-1}-\Gamma_{k,j-1})
(\overline{\beta}_j-\beta_{j-1})-(\alpha_k-\overline{\alpha}_{k+1})
\Gamma_{k,j-1}\right].\label{feb6e}
\end{align}
For $j\ge 3$, we use equality \eqref{feb6} with the extra term $\Delta_{k+1,j-2}$
on the right; substituting this modified equality into \eqref{feb6e} gives
the following modification of \eqref{4.27}:
\begin{align}
\Delta_{k,j}=&(\overline{\alpha}_{k+1}-\alpha_k)^{-1}\left[
d_{k,j}-c_{k-1,j}+\overline{\alpha}_{k}\Gamma_{k-1,j}-\Gamma_{k-1,j}\beta_j
-\Gamma_{k-1,j-1}\right.\notag\\
&\left. +(\Delta_{k+1,j-1}-\Gamma_{k,j-1})                 
(\overline{\beta}_j-\beta_{j-1})+\Gamma_{k,j-2}-\Delta_{k+1,j-2}
\right].\label{feb6f}
\end{align}
Combining the latter equality with \eqref{feb6d} (in case $k\ge 2$)
gives \eqref{4.6fb}.
\end{proof}
Since $\Gamma_{1,1}=(\alpha_1-\overline{\alpha}_{2})^{-1}d_{2,1}$, we derive from 
\eqref{feb6e} 
$$
\Delta_{1,2}=(\overline{\alpha}_2-\alpha_1)^{-1}\left[
d_{1,2}+(\Delta_{2,1}-\Gamma_{1,1})(\overline{\beta}_2-\beta_{1})\right],
$$
which, on account of formulas \eqref{4.11} and \eqref{feb7} for $S_2$ and $\Delta_{2,1}$,
respectively, can be written as 
$$
\Delta_{1,2}=(\overline{\alpha}_2-\alpha_1)^{-1}S_2+(\overline{\alpha}_2-\alpha_1)^{-1}
(\overline{\alpha}_3-\alpha_2)^{-1}
(\overline{\alpha}_2-\alpha_1)^{-1}S_1(\overline{\beta}_2-\beta_{1}).
$$
As expected, the letter formula coincides with \eqref{4.23} in case $S_1=0$. For $j\ge 3$,
we have from \eqref{feb6f} and formula \eqref{4.11} for $S_j$,
\begin{align}
\Delta_{1,j}=&(\overline{\alpha}_{2}-\alpha_1)^{-1}\left[
d_{1,j}+(\Delta_{2,j-1}-\Gamma_{1,j-1})
(\overline{\beta}_j-\beta_{j-1})+\Gamma_{1,j-2}-\Delta_{2,j-2}
\right]\notag\\
=&(\overline{\alpha}_{2}-\alpha_1)^{-1}\left[S_j+\Delta_{2,j-1}(\overline{\beta}_j-\beta_{j-1})
-\Delta_{2,j-2}\right],
\label{feb7a}
\end{align}
which is the analog of \eqref{4.28}. It is now clear from \eqref{4.6fb} and \eqref{feb7a} that 
$\Delta_{k,j}$ is a two-sided linear combination of $S_1,\ldots,S_k$ with left and right coefficients 
depending respectively, on $\balpha$ and $\bbeta$ only. However, explicit formulas for 
$\Delta_{k,j}$ in terms of $S_1,\ldots,S_k$ are quite complicated.	
\begin{theorem}
Given matrices $A$, $B$ (based on spherical chains $\balpha=(\alpha_1,\ldots,\alpha_n)$ and 
$\bbeta=(\beta_1,\ldots,\beta_m)$ ($m\le n$)from the same conjugacy 
class $V\subset\bH$) and $C$ as in \eqref{4.6}, let $d_{i,j}$ and $\Gamma_{i,j}$
be defined as in \eqref{4.7}--\eqref{4.10} for $k<n$ and
$j\le m$  such that $k+j\le n$. Then 
\begin{enumerate}
\item The Sylvester equation \eqref{4.6} has a solution if and only if 
\begin{equation}
d_{1,j}+\Gamma_{1,j-1}(\beta_{j-1}-\overline{\beta}_j)+\Gamma_{1,j-2}=0\quad\mbox{for}\quad 
j=1,\ldots,m.
\label{4.30}
\end{equation}
\item For any solution $X=\left[x_{i,j}\right]\in\bH^{n\times m}$ to the equation \eqref{4.6},
$x_{i,j}=\Gamma_{i,j}$ for all $i<n$ and $j\le m$  such that $i+j\le n$
\end{enumerate}
\label{T:4.4}
\end{theorem}
\begin{proof}
Let $X\in\bH^{n\times m}$ satisfy \eqref{4.6}. We now verify that 
\begin{equation}
\widetilde{A}X-X\widetilde{B}=D
\label{4.31}
\end{equation}
where $A^\prime$, $\widetilde{A}$ and $D$ are defined in \eqref{3.4a}, \eqref{3.4c} and \eqref{4.7}, 
respectively, 
and where
\begin{equation}
\widetilde{B}=
\left[\begin{array}{cccccc} 0 &\beta_1-\overline{\beta}_2 & 1 & 0 &\ldots&0\\
0 & 0 & \beta_2-\overline{\beta}_3 &\ddots & \ddots & \vdots\\
 &&&&& 0\\
\vdots&\vdots&\ddots&\ddots&\ddots& 1\\
&&&&0& \beta_{m-1}-\overline{\beta}_m \\
0&0&\ldots &&0& 0
\end{array}\right].
\label{4.32}
\end{equation}
We start as in the proof of Theorem \ref{T:3.2a}:  multiplying the equation \eqref{4.6} by $A^\prime$
on the  left and by $B$ on the right and then subtracting the second equality from
the first we get \eqref{3.4d}. For $B$ of the form \eqref{4.1},
\begin{align*}
\cX_V(B)&=B^2-2{\rm Re}\beta_1 \cdot
B+|\beta_1|^2I_m\\&=\left[\cX_V(\beta_i)\delta_{i,j}+(\beta_i-\overline{\beta}_{i+1})\delta_{i+1,j}+
\delta_{i+2,j}\right]_{i,j=1}^m,
\end{align*}
and since $\beta_i\in V$ for $i=1,\ldots,n$, we conclude that $\cX_V(B)=\widetilde{B}$; see \eqref{4.32}.
Now \eqref{3.4d} takes the form \eqref{4.31}.

\smallskip

Let $X_j$ and $D_j$ denote the $j$-th column in $X$ and $D$, respectively:
$$
X_j={\rm Col}_{1\le k\le n}x_{k,j},\quad D_j={\rm Col}_{1\le k\le n}d_{k,j}.
$$
Taking into account the explicit structure \eqref{4.32} of $\widetilde{B}$, 
we now equate the corresponding columns in \eqref{4.31}:
\begin{align}
\widetilde AX_1&=D_1, \label{4.33}\\
\widetilde AX_2&=D_2+X_1(\beta_1-\overline{\beta}_2),\label{4.34}\\
\widetilde AX_j&=D_j+X_{j-1}(\beta_{j-1}-\overline{\beta}_j)+X_{j-2}\quad (3\le j\le m).
\label{4.35}
\end{align}
Making use of explicit formula \eqref{3.4c} for $\widetilde{A}$ we equate the top entries in \eqref{4.33} and 
get
$d_{1,1}=0$ (that is, the first condition in \eqref{4.30}). Equating other corresponding entries in
\eqref{4.33} we get
$$
\mathbb A_n \begin{bmatrix}x_{1,1} \\ x_{2,1} \\ \vdots \\ x_{n-1,1}
\end{bmatrix}=\begin{bmatrix}d_{2,1} \\ d_{3,1} \\ \vdots \\ d_{n,1}
\end{bmatrix},
$$
where 
\begin{equation}
\mathbb A_n =\left[\begin{array}{cccccc}
\alpha_1-\overline{\alpha}_2&0&&&0\\ 1 &\alpha_2-\overline{\alpha}_3&\ddots&\ddots&\vdots\\
0&1& \ddots&\ddots&\vdots \\
\vdots&\ddots&\ddots&\ddots&0\\ 0&\ldots &0&1&\alpha_{n-1}-\overline{\alpha}_n\end{array}\right]
\label{4.16}
\end{equation}
from which we conclude
\begin{align}
x_{1,1}&=(\alpha_1-\overline{\alpha}_{2})^{-1}d_{2,1},\notag\\
x_{k,1}&=(\alpha_k-\overline{\alpha}_{k+1})^{-1}\left(d_{k+1,1}-x_{k-1,1}
\right), \quad 2\le k\le  n-1.\label{4.36}
\end{align}
Comparing \eqref{4.36} with \eqref{4.14}, we see that $x_{1,1}=\Gamma_{1,1}$ and that recursions
defining  $x_{k,1}$ and $\Gamma_{k,1}$ in terms of $x_{k-1,1}$ and $\Gamma_{k-1,1}$, respectively,
are identical. Therefore, 
\begin{equation}
x_{k,1}=\Gamma_{k,1}\quad\mbox{for}\quad k=1,\ldots,n-1.
\label{4.37}
\end{equation}
Making again use of formula \eqref{3.4c} for $\widetilde{A}$ we equate the top entries in \eqref{4.34} and
get 
$$
d_{1,2}+x_{1,1}(\beta_1-\overline{\beta}_2)=d_{1,2}+\Gamma_{1,1}(\beta_1-\overline{\beta}_2)=0
$$
(the second equality holds since $x_{1,1}=\Gamma_{1,1}$, by \eqref{4.37}), which is the second condition in  
\eqref{4.30}. Equating other corresponding entries in
\eqref{4.33} and taking into account \eqref{4.37}, we get
$$
\mathbb A_n \begin{bmatrix}x_{1,2} \\ \vdots \\ x_{n-2,2}\\ x_{n-1,2}
\end{bmatrix}=\begin{bmatrix}d_{2,1} \\ \vdots \\ d_{n-1,1} \\ d_{n,1}
\end{bmatrix}+\begin{bmatrix}x_{2,1} \\ \vdots \\ x_{n-1,1} \\ x_{n,1}
\end{bmatrix}(\beta_1-\overline{\beta}_2)=\begin{bmatrix}d_{2,1} \\ \vdots \\ d_{n-1,1} \\ d_{n,1}
\end{bmatrix}+\begin{bmatrix}\Gamma_{2,1} \\ \vdots \\ \Gamma_{n-1,1} \\ x_{n,1}
\end{bmatrix}(\beta_1-\overline{\beta}_2),
$$
where $\mathbb A_n$ is given by \eqref{4.16}. Taking into account the two-diagonal structure 
\eqref{4.16}
of  $\mathbb A_n$, we derive from the last equation
\begin{align}
x_{1,2}&=(\alpha_1-\overline{\alpha}_{2})^{-1}(d_{2,1}+\Gamma_{1,1}(\beta_1-\overline{\beta}_2)),\label{4.38}\\
x_{k,2}&=(\alpha_k-\overline{\alpha}_{k+1})^{-1}\left(d_{k+1,1}+
\Gamma_{k+1,1}(\beta_1-\overline{\beta}_2)-\Gamma_{k-1,2}
\right), \quad 2\le k\le  n-2.\notag
\end{align}
Comparing \eqref{4.38} with \eqref{4.9}, we see that $x_{1,2}=\Gamma_{1,2}$ and that recursions
defining  $x_{k,2}$ and $\Gamma_{k,2}$ in terms of $x_{k-1,2}$ and $\Gamma_{k-1,2}$, respectively,
are identical. Therefore,
\begin{equation}
x_{k,2}=\Gamma_{k,2}\quad\mbox{for}\quad k=1,\ldots,n-2.
\label{4.39}
\end{equation}
We now choose an integer $\ell$ ($3\le \ell<m$) and assume that 
conditions \eqref{4.30} hold for all $j=1,\ldots,\ell-1$ and that 
$x_{k,j}=\Gamma_{k,j}$ for all $j<\ell$ and $k=1,\ldots,n-j$. Equating the top entries in \eqref{4.35}
(for $j=\ell$) we get 
$$
0=d_{1,\ell}+x_{1,\ell-1}(\beta_{\ell-1}-\overline{\beta}_\ell)+x_{1,\ell-2}=
d_{1,\ell}+\Gamma_{1,\ell-1}(\beta_{\ell-1}-\overline{\beta}_\ell)+\Gamma_{1,\ell-2}
$$ 
which is the $\ell$-th condition in \eqref{4.30}. Equating all other entries in the same equation gives
$$
\mathbb A_n \begin{bmatrix}x_{1,\ell} \\ \vdots \\ x_{n-1,\ell}
\end{bmatrix}=\begin{bmatrix}d_{2,\ell} \\ \vdots \\ d_{n,\ell}
\end{bmatrix}+\begin{bmatrix}x_{2,\ell-1} \\ \vdots \\ 
x_{n,\ell-1}\end{bmatrix}(\beta_{\ell-1}-\overline{\beta}_\ell)
+\begin{bmatrix}x_{2,\ell-2} \\ \vdots \\ x_{n,\ell-2}\end{bmatrix}
$$
which implies (similarly to \eqref{4.36} and \eqref{4.38},
\begin{align*}
x_{1,\ell}&=(\alpha_\ell-\overline{\alpha}_{\ell+1})^{-1}(d_{2,\ell}+
x_{2,\ell-1}(\beta_1-\overline{\beta}_2)+x_{2,\ell-2}),\\
x_{k,\ell}&=(\alpha_k-\overline{\alpha}_{k+1})^{-1}\left(d_{k+1,\ell}+
x_{k+1,\ell-1}(\beta_{\ell-1}-\overline{\beta}_\ell)+x_{k+1,\ell-2}-x_{k-1,\ell}\right)
\end{align*}
for $k=2,\ldots,n$. At least for $k\le n-\ell$, we may use assumptions 
$x_{k,j}=\Gamma_{k,j}$ to rewrite the above equalities (more precisely, the right
hand side expressions in the above equalities) as 
\begin{align}
x_{1,\ell}&=(\alpha_\ell-\overline{\alpha}_{\ell+1})^{-1}(d_{2,\ell}+
\Gamma_{2,\ell-1}(\beta_1-\overline{\beta}_2)+\Gamma_{2,\ell-2}),\label{4.40}\\
x_{k,\ell}&=(\alpha_k-\overline{\alpha}_{k+1})^{-1}\left(d_{k+1,\ell}+
\Gamma_{k+1,1}(\beta_{\ell-1}-\overline{\beta}_\ell)+\Gamma_{k+1,\ell-2}-x_{k-1,\ell}\right).\notag
\end{align}      
We again compare \eqref{4.39} with \eqref{4.9} to see that $x_{1\ell}=\Gamma_{1,\ell}$ and 
that $x_{k,\ell}$ and $\Gamma_{k,\ell}$ are defined by the same recursion. Therefore 
$x_{k,\ell}=\Gamma_{k,\ell}$ for $k=1,\ldots,n-\ell$, and then we conclude by induction 
that equalities \eqref{4.30} hold for $j=1,\ldots,m$ and that $x_{k,j}=\Gamma_{k,j}$ for all $k<n$ and $j<m$  such 
that $k+j\le n$.

\smallskip

It remains to show that conditions \eqref{4.30} are sufficient for the equation \eqref{4.6} to have a solution.
Let us extend the spherical chain $\balpha=(\alpha_1,\ldots,\alpha_n)$ to a spherical chain
$\widehat\balpha=(\alpha_1,\ldots,\alpha_n,\alpha_{n+1},\ldots,\alpha_{n+m})$. We may let, for example, 
$\alpha_{n+i}=\alpha_n$ for $i=1,\ldots,m$. We then let 
$$
\widehat A={\mathcal J}_{\widehat \balpha}=
\left[\alpha_i\delta_{i,j}+\delta_{i,j+1}\right]_{i,j=1}^{n+m}=\begin{bmatrix}A & 0 \\ * & *\end{bmatrix}
$$
and 
$$
\widehat A^\prime=
\left[\overline{\alpha}_i\delta_{i,j}-\delta_{i,j+1}\right]_{i,j=1}^{n+m}=\begin{bmatrix}A^\prime & 0 \\ * & 
*\end{bmatrix}
$$
to be the corresponding extensions of $A$ and $A^\prime$ (see \eqref{3.4a}) and we arbitrarily extend the given 
matrix $C$ to 
$$
\widehat C=\begin{bmatrix}C \\ 
C^\prime\end{bmatrix}=\left[c_{i,j}\right]_{i=1,\ldots,n+m}^{j=1,\ldots,m}\in\bH^{(n+m)\times m}
$$
(we may choose $C^\prime=0$). We then consider the matrix $\widehat D^\prime=\widehat CB-\widehat A^\prime\widehat 
C\in\bH^{(n+m)\times n}$ 
which 
is indeed an extension of the matrix \eqref{4.7}:
\begin{align*}
\widehat D^\prime=\widehat CB-\widehat A^\prime\widehat C&=
\begin{bmatrix}C \\ C^\prime\end{bmatrix}B-\begin{bmatrix}A^\prime & 0 \\ * &
*\end{bmatrix}\begin{bmatrix}C \\ C^\prime\end{bmatrix}\\
&=\begin{bmatrix}CB-A^\prime C \\ D^\prime\end{bmatrix}
=\begin{bmatrix}D \\ D^\prime\end{bmatrix}=\left[d_{i,j}\right]_{i=1,\ldots,n+m}^{j=1,\ldots,m}.
\end{align*}
We now use recursions \eqref{4.9}, \eqref{4.10} to define the elements $\Gamma_{k,j}$ for all 
positive $k<n+m$ and $j\le m$ such that $k+j\le n+m$; see Remark \ref{R:4.1} (part (1)). 
By Remark \ref{R:4.1} (part (3)) and by formula \eqref{4.9}, the elements $\Gamma_{k,j}$ (for $k+j\le n$) are 
completely determined by the entries of the original matrices $A$, $B$ and $C$. Since equalities \eqref{4.30}
hold (i.e., the elements $S_j$ in \eqref{4.11} are equal to zero for $j=1,\ldots,m$, Lemma \ref{L:6.3} 
applies
to the extended set $\{\Gamma_{k,j}: \, k+j\le n+m\}$ and guarantees that equalities \eqref{4.12} hold for 
all $k<n+m$ and $j\le m$ such that $k+j\le n+m$. In particular, equalities \eqref{4.12} hold for 
all $1\le k\le n$ and $1\le j\le m$. Furthermore, if we let 
\begin{equation}
\Gamma=\left[\Gamma_{k,j}\right]_{k=1,\ldots,n}^{j=1,\ldots,m}\in\bH^{n\times m},
\label{4.41}
\end{equation}
then (as it is readily seen from \eqref{4.1}) $\Delta_{k,j}$ defined as in \eqref{4.12} is the $(k,j)$-entry
in the matrix
$$
\Delta=\left[\Delta_{k,j}\right]_{k,j}=A\Gamma-\Gamma B-C.
$$
Equalities \eqref{4.12} for all $1\le k\le n$ and $1\le j\le m$ mean that $\Delta=0$, i.e., that 
$\Gamma$ is a solution to the equation \eqref{4.6}.
\end{proof}
The last theorem suggested an algorithm for constructing a particular solution for a solvable
singular Sylvester equation. We record it for the convenience of future reference.
\begin{algorithm}
{\rm Given $A$ and $B$ of the form \eqref{4.1} where $\balpha=(\alpha_1,\ldots,\alpha_n)$ and 
$\bbeta=(\beta_1,\ldots,\beta_n)$ ($m\le n$) are spherical chains from the same conjugacy class, and given 
$C=[c_{i,k}]\in\bH^{n\times m}$,
\begin{enumerate}
\item Compute $d_{i,j}$ for $i+j\le n$ by formula \eqref{4.8}.
\item Compute $\Gamma_{i,j}$ for $i+j\le n$ by formulas \eqref{4.9}, \eqref{4.10}.
\item Verify equalities \eqref{4.11} for $j=1,\ldots,m$. If they hold true, 
equation \eqref{4.6} has a solution.
\item Let $\alpha_{n+i}=\alpha_{n}$ and $c_{n+i,j}=0$ for $i,j=1,\ldots,m$.
\item Compute $d_{i,j}$ for $n+1\le i+j\le n+m$ by formula \eqref{4.8}.
\item Compute $\Gamma_{i,j}$ for $n+1\le i+j\le n+m$ by formula \eqref{4.9}. 
\end{enumerate}
The matrix \eqref{4.41} is a particular solution of the Sylvester equation \eqref{4.6}}
\label{A:4.5}
\end{algorithm} 
It is of particular interest to write necessary and sufficient conditions 
\eqref{4.30} for solvability of a singular Sylvester equation \eqref{4.6} 
in terms of given matrices $A$, $B$ and $C$. We were able to establish such formulas
only in the case where $A$ and $B$ are Jordan blocks (i.e., $A$ and $B$ are based on spherical 
chains of the form \eqref{2.3}). Theorem \ref{T:4.6} below is the quaternion analog of 
a result of Ma \cite{ma}; see \eqref{1.7}.
\begin{theorem}
Let $A$ and $B$ be of the form \eqref{4.1} with $\alpha_i=\alpha\sim\beta=\beta_j$
($1\le i\le n$, $1\le j\le m\le n$). Given a matrix 
$C=[c_{k,j}]\in\bH^{n\times m}$, let 
\begin{align*}
d_{1,1}&=c_{1,1}\beta-\overline{\alpha}c_{1,1},\\
d_{1,j}&=c_{1,j}\beta-\overline{\alpha}c_{1,j}+c_{1,j-1} \; \; (j\ge 2),\\
d_{k,1}&=c_{k,1}\beta-\overline{\alpha}c_{k,1}+c_{k-1,1} \; \; (k\ge 2),\\
d_{k,j}&=c_{k,j}\beta-\overline{\alpha}c_{k,j}+c_{k-1,j}+c_{k,j-1} \; \; (k,j\ge 2).
\end{align*} 
Sylvester equation \eqref{4.6} has a solution if and only (c.f. with \eqref{1.7})
\begin{equation}
\sum_{2\ell\le j}{\rm Im}(\alpha)\cdot d_{2\ell+1,j-2\ell}
+\sum_{2\ell\le j}d_{2\ell,j-2\ell+1}\cdot{\rm Im}(\beta)=0
\quad\mbox{for} \quad j=1,\ldots,m.
\label{4.42}
\end{equation}
\label{T:4.6}
\end{theorem}
\begin{proof}
We will show that in the present setting, conditions \eqref{4.30} are equivalent to 
\eqref{4.42}. We fist observe that $|{\rm Im}(\alpha)|=|{\rm Im}(\beta)|$ by characterization 
\eqref{2.1}.
Therefore, $(\alpha-\overline{\alpha})^2=(\beta-\overline{\beta})^2$ is a negative number
and therefore 
\begin{equation}
(\alpha-\overline{\alpha})^{-2}d(\beta-\overline{\beta})^2=d \quad\mbox{for any}\quad d\in\bH.
\label{4.43}
\end{equation}
Making use of notation \eqref{4.11} and formula \eqref{4.9} for $\Gamma_{1,j-1}$ (specified to the 
present particular case) we get
\begin{align}
(\alpha-\overline{\alpha})S_j=&
(\alpha-\overline{\alpha})\left[ 
d_{1,j}+\Gamma_{1,j-1}(\beta-\overline{\beta})+\Gamma_{1,j-2}\right]\notag\\
=&(\alpha-\overline{\alpha})(d_{1,j}+\Gamma_{1,j-2})+
\left(d_{2,j-1}+\Gamma_{2,j-3}+\Gamma_{2,j-2}(\beta-\overline{\beta})\right)(\beta-\overline{\beta})
\notag\\
=&(\alpha-\overline{\alpha})d_{1,j}+d_{2,j-1}(\beta-\overline{\beta})\notag\\
&+
\Gamma_{2,j-3}(\beta-\overline{\beta})+(\alpha-\overline{\alpha})^2\Gamma_{2,j-2}+
(\alpha-\overline{\alpha})\Gamma_{1,j-2},\label{4.44}
\end{align}
where the last equality holds due to \eqref{4.43}. 
Making subsequent use of formula \eqref{4.9} for $\Gamma_{2,j-2}$ and then for 
$\Gamma_{3,j-3}$, we conclude that 
the sum of the three rightmost terms in \eqref{4.44} equals
\begin{align*}
&\Gamma_{2,j-3}(\beta-\overline{\beta})+
(\alpha-\overline{\alpha})\left[d_{3,j-2}+\Gamma_{3,j-4}
+\Gamma_{3,j-3}(\beta-\overline{\beta})\right]\notag\\
&=(\alpha-\overline{\alpha})\left[d_{3,j-2}+\Gamma_{3,j-4}\right]
+\left[d_{4,j-3}+\Gamma_{4,j-5}+\Gamma_{4,j-4}(\beta-\overline{\beta})\right](\beta-\overline{\beta})
\notag\\
&=(\alpha-\overline{\alpha})d_{3,j-2}+d_{4,j-3}(\beta-\overline{\beta})\notag\\ 
&\quad +
\Gamma_{4,j-5}(\beta-\overline{\beta})+(\alpha-\overline{\alpha})^2\Gamma_{4,j-4}+
(\alpha-\overline{\alpha})\Gamma_{3,j-4}.
\end{align*}
which being substituted into \eqref{4.44} gives
\begin{align}
(\alpha-\overline{\alpha})S_j=&(\alpha-\overline{\alpha})(d_{1,j}+d_{3,j-2}) 
+(d_{2,j-1}+d_{4,j-3})(\beta-\overline{\beta})\notag\\
&+\Gamma_{4,j-5}(\beta-\overline{\beta})+(\alpha-\overline{\alpha})^2\Gamma_{4,j-4}+
(\alpha-\overline{\alpha})\Gamma_{3,j-4}.\label{4.45}            
\end{align}
The step which led us from \eqref{4.44} to \eqref{4.45} can be repeated indefinitely.
After $\ell$ iterations we get 
\begin{align} 
(\alpha-\overline{\alpha})S_j=&(\alpha-\overline{\alpha})\sum_{k=0}^{\ell-1}d_{2k+1,j-2k}
+\sum_{k=1}^{\ell}d_{2k,j-2k+1}(\beta-\overline{\beta})\label{4.46}\\
&+\Gamma_{2\ell,j-2\ell-1}(\beta-\overline{\beta})+(\alpha-\overline{\alpha})^2\Gamma_{2\ell,j-2\ell}+
(\alpha-\overline{\alpha})\Gamma_{2\ell-1,j-2\ell}.\notag
\end{align}
Recall that $\Gamma_{k,j}$ and $d_{k,j}$ are equal zero whenever one of the indices is 
nonpositive. If we choose $\ell$ be greater than $j/2$ in \eqref{4.45} and take into account only 
nonzero terms, we conclude that $(\alpha-\overline{\alpha})S_j$ equals the expression 
on the left side of \eqref{4.42}. Thus, conditions \eqref{4.42} are equivalent (since $\alpha\neq 
\overline{\alpha}$ to those in \eqref{4.30} and applying Theorem \ref{T:4.4} completes the proof.
\end{proof}

\section{Singular case: parametrization of all solutions}
\setcounter{equation}{0}

Since the equation \eqref{4.6} is linear, all its solutions $X$ are described by 
$X=\Gamma+Y$, where $Y$ is the general solution of the homogeneous equation
\begin{equation}
AY-YB=0. 
\label{5.1}
\end{equation}
we will use notation $\mathcal D_\ell(Y)=\{Y_{k,j}: \, k+j=\ell+1\}$  
for the $\ell$-th counter-diagonal of the matrix $Y=[Y_{k,j}]$. 

\smallskip

In this final part of the paper, we will present a parametrization of the solution set of the 
homogeneous equation \eqref{5.1} with free parameters $\mu_j\in\Pi_{\alpha_n,\beta_j}$
($j=1,\ldots,m$) from the planes $\Pi_{\alpha_n,\beta_j}\subset \bH$ defined as in \eqref{4.4}.
The latter memberships mean, by Lemma \ref{L:6.2}, that 
\begin{equation}
\alpha_n\mu_j=\mu_j\beta_j \quad\mbox{for}\quad j=1,\ldots,m.
\label{5.2} 
\end{equation}
For any fixed tuple ${\boldsymbol \mu}=(\mu_1,\ldots,\mu_m)$ of such parameters, we recursively define 
the entries $Y_{k,j}$ of the matrix $Y=[Y_{k,j}]\in\bH^{n\times m}$ by the formulas
\begin{align}
Y_{k,j}=&0\quad\mbox{for all}\quad k+j\le n,\quad Y_{n,1}=\mu_1,\label{5.3}\\
Y_{k,j}=&(\overline{\alpha}_{k+1}-\alpha_k)^{-1}\left[Y_{k+1,j-1}(\overline{\beta}_j-\beta_{j-1})
-Y_{k+1,j-2}+Y_{k-1,j}\right]\label{5.4}\\
&(n-m+1\le k\le n-1; \; \; n-k+1\le j\le n+m-k),\notag\\
Y_{n,j+1}=&\mu_{j+1}-\sum_{\ell=1}^j \left(\prod_{i=1}^{\substack{\curvearrowright \\ \ell}}
(\overline{\alpha}_{n-i+1}-\alpha_{n-i})^{-1}\right)Y_{n-\ell+1,j}\quad (2\le j\le m).\label{5.5}
\end{align}
For the sake of clarity, we display formula \eqref{5.5} in a less compact form
\begin{align*}
Y_{n,j+1}=&\mu_{j+1}-(\overline{\alpha}_n-\alpha_{n-1})^{-1}Y_{n,j}\notag\\
&-(\overline{\alpha}_n-\alpha_{n-1})^{-1}(\overline{\alpha}_{n-1}-\alpha_{n-2})^{-1}Y_{n-1,j}-\ldots\notag\\
&-(\overline{\alpha}_n-\alpha_{n-1})^{-1}\cdots
(\overline{\alpha}_{n-j+1}-\alpha_{n-j})^{-1}Y_{n-j+1,j}.
\end{align*}
The construction of $Y$ goes as follows: with \eqref{5.3} in hands, we use formula \eqref{5.4} to compute all 
entries in the counter-diagonal $\mathcal D_{n+1}(Y)$, it is not hard to see that
\begin{equation}
Y_{n-j,j+1}=\left(\prod_{i=1}^{\substack{\curvearrowleft \\ j}}
(\overline{\alpha}_{n-i+1}-\alpha_{n-i})^{-1}\right)\mu_1 \left(\prod_{i=1}^{\substack{\curvearrowright \\ j}}   
(\overline{\beta}_{i+1}-\beta_{i})\right).
\label{5.5a}
\end{equation}
In particular, for $j=2$ we have 
\begin{equation}
Y_{n-1,2}=(\overline{\alpha}_n-\alpha_{n-1})^{-1}\mu_1 \, (\overline{\beta}_{2}-\beta_{1})
\label{5.6}
\end{equation}
and then we use \eqref{5.5} to compute
\begin{equation}
Y_{n,2}=\mu_2-(\overline{\alpha}_n-\alpha_{n-1})^{-1}Y_{n,1}=\mu_2-(\overline{\alpha}_n-\alpha_{n-1})^{-1}\mu_{1},
\label{5.7}
\end{equation}
and then again \eqref{5.4} to compute the entries in $\mathcal D_{n+2}(Y)$. Once all 
counter-diagonals $\mathcal D_{\ell}(Y)$ are found for $r=2,\ldots,n+j+1$
(and in particular, $Y_{k,j}$ are specified for all $k=1,\ldots,n$), we use \eqref{5.5} to find 
$Y_{n,j+1}$ and then \eqref{5.4}to compute the next counter-diagonals $\mathcal D_{n+j+2}(Y)$.
\begin{theorem}
Let ${\boldsymbol \mu}=(\mu_1,\ldots,\mu_m)$ satisfy \eqref{5.2}. Then the matrix $Y\in\bH^{n\times m}$
constructed as in \eqref{5.3}--\eqref{5.5} is a solution to the equation \eqref{5.1}. Moreover, any 
solution $Y$ to the equation \eqref{5.1} arises in this way.
\label{T:5.1}
\end{theorem}
\begin{proof}
The recursion formula \eqref{5.4} is the homogeneous version of \eqref{4.9} (with all $c_{k,j}$
or, equivalently, all $d_{k,j}$ equal zero). Hence, some conclusions will be obtained 
by applying the corresponding counter-parts from the previous section. In analogy to the non-homogeneous 
case, we introduce 
\begin{equation}
\Delta_{k,j}:=\alpha_kY_{k,j}-Y_{k,j}\beta_j+Y_{k-1,j}-Y_{k,j-1}
\label{5.8}
\end{equation}
for $k=1,\ldots,n$ and $j=1,\ldots,m$, and recall relations 
\eqref{4.6fb} holding for all $k\ge 3$ and $j\ge 2$. We write them
equivalently as 
\begin{equation}
\Delta_{k-1,j}=(\overline{\alpha}_{k+1}-\alpha_k)\Delta_{k,j}-\Delta_{k+1,j-1}
(\overline{\beta}_j-\beta_{j-1})+\Delta_{k+1,j-2}
\label{5.9}
\end{equation}
and observe that $\Delta_{k+1,j-2}$ and $\Delta_{k-1,j}$ are the consequtive entries 
on the $(k+j+1)$-th counter-diagonal of the matrix $\Delta=[\Delta_{k,j}]$ whereas 
$\Delta_{k+1,j-1}$ and $\Delta_{k+1,j-2}$ are the entries from the ``previous"
counter-diagonal $\mathcal D_{k+j}(\Delta)$. Thus, relations \eqref{5.9} guarantee that once 
all entries in $\mathcal D_{r}(\Delta)$ are zeros for some $r>n$ and the bottom entry of the
next counter-diagonal $\mathcal D_{r+1}(\Delta)$ is zero, then all entries in $\mathcal D_{r+1}(\Delta)$
are zeros. In other words, the following implication holds true: 
{\em for any fixed $r$ ($n< r<n+m$),
\begin{equation}
\left\{\begin{array}{l}
\Delta_{k,r-k}=0 \; (r-m\le k\le n),\\
\Delta_{n,r+1-n}=0
\end{array}\right. \; \Longrightarrow \; 
\Delta_{k,r+1-k}=0 \; (r-m+1\le k\le n). 
\label{5.10}
\end{equation}}
Furthermore, due to \eqref{5.3}, 
$\Delta_{k,j}=0$ for all $(k,j)$ such that $k+j\le n$. We next observe that due to \eqref{5.2},
\begin{equation}
\Delta_{n,1}=\alpha_n\mu_1-\mu_1\beta_1=0.
\label{5.11}
\end{equation}
Making use of \eqref{5.6}, \eqref{4.13} and \eqref{5.11} we get
\begin{align*}
\Delta_{n-1,2}&=\alpha_{n-1}Y_{n-1,2}-Y_{n-1,2}\beta_2\notag\\
&=\alpha_{n-1}(\overline{\alpha}_n-\alpha_{n-1})^{-1}\mu_1 \, (\overline{\beta}_{2}-\beta_{1})
-(\overline{\alpha}_n-\alpha_{n-1})^{-1}\mu_1 \, (\overline{\beta}_{2}-\beta_{1})\beta_2\notag\\
&=(\overline{\alpha}_n-\alpha_{n-1})^{-1}\left(\alpha_n\mu_1-\mu_1\beta_1\right)
(\overline{\beta}_{2}-\beta_{1})=0.
\end{align*}
Since $\Delta_{k,j}=0$ for all $k,j$ such that $k+j=n$ and since $\Delta_{n-1,2}=0$, we recursively get from 
\eqref{5.9} that  $\Delta_{k,j}=0$ for all $k,j$ such that $k+j=n+1$ (i.e., that all the entries in 
$\mathcal D_{n}(\Delta)$ are zeros).  To use induction principle, let us assume that 
$\Delta_{k,i}=0$ for all $k,i$ such that $k+i\le r$ for some $r\ge n+2$ and show that $\Delta_{k,i}=0$
also for all $k,i$ such that $k+i= r+1$. Let $j:=n-r$. By the assumption, we have in particular,
$\Delta_{k,j}=0$ for all $k=1,\ldots,n$ which implies, due to formula \eqref{5.8}, equalities
\begin{equation}
\alpha_kY_{k,j}-Y_{k,j}\beta_j=Y_{k,j-1}-Y_{k-1,j}\quad\mbox{for}\quad k=2,\ldots,n.
\label{5.13}
\end{equation}
Since ${\rm Re}(\alpha_k)={\rm Re}(\beta_j)$ (by characterization \eqref{2.1}), the latter equalities can be written 
equivalently as 
\begin{equation}
\overline{\alpha}_kY_{k,j}-Y_{k,j}\overline{\beta}_j=Y_{k-1,j}-Y_{k,j-1}\quad\mbox{for}\quad k=2,\ldots,n.
\label{5.14}
\end{equation}
We next compute, for a fixed $k$ ($2\le k\le n-1)$, 
\begin{align}
&Y_{k,j+1}-Y_{k+1,j}\notag\\
&=(\overline{\alpha}_{k+1}-\alpha_{k})^{-1}\left[Y_{k+1,j}(\overline{\beta}_{j}-\beta_{j+1})
-Y_{k+1,j-1}+Y_{k-1,j+1}\right.\notag\\
&\qquad\qquad\qquad\qquad \left.-(\overline{\alpha}_{k+1}-\alpha_{k})Y_{k+1,j}\right]\notag\\
&=(\overline{\alpha}_{k+1}-\alpha_{k})^{-1}\left[\alpha_{k}Y_{k+1,j}-Y_{k+1,j}\beta_{j+1}
-Y_{k,j}+Y_{k-1,j+1}\right]\notag\\
&=\alpha_{k+1}(\overline{\alpha}_{k+1}-\alpha_{k})^{-1}Y_{k+1,j}-(\overline{\alpha}_{k+1}-\alpha_{k})^{-1}
Y_{k+1,j}\beta_{j+1}\notag\\
&\qquad+(\overline{\alpha}_{k+1}-\alpha_{k})^{-1}\left[Y_{k-1,j+1}-Y_{k,j}\right],
\end{align}
where for the first equality  we used formula \eqref{5.4} for $Y_{k,j+1}$, the second equality 
holds due to \eqref{5.14} (with $k$ replaced by $(k+1)$), and the last equality relies on \eqref{4.13} 
(for $\alpha=\alpha_{k+1}$ and $\beta=\alpha_{k}$). Writing formula \eqref{5.14} for $k=n-1$ and then
iterating it once and making use of \eqref{4.13}, we get 
\begin{align}
Y_{n-1,j+1}-Y_{n,j}=&\alpha_{n}(\overline{\alpha}_{n}-\alpha_{n-1})^{-1}Y_{n,j}-(\overline{\alpha}_{n}-\alpha_{n-1})^{-1}
Y_{n,j}\beta_{j+1}\notag\\
&+(\overline{\alpha}_{n}-\alpha_{n-1})^{-1}
\left[\alpha_{n-1}(\overline{\alpha}_{n-1}-\alpha_{n-2})^{-1}Y_{n-1,j}\right.\notag\\
&\left.-(\overline{\alpha}_{n-1}-\alpha_{n-2})^{-1}
Y_{n-1,j}\beta_{j+1}\right.\notag\\
&\left.+(\overline{\alpha}_{n-1}-\alpha_{n-2})^{-1}\left[Y_{n-3,j+1}-Y_{n-2,j}\right]\right]\notag\\
=&\alpha_{n}(\overline{\alpha}_{n}-\alpha_{n-1})^{-1}Y_{n,j}-(\overline{\alpha}_{n}-\alpha_{n-1})^{-1}
Y_{n,j}\beta_{j+1}\notag\\
&+\alpha_{n}(\overline{\alpha}_{n}-\alpha_{n-1})^{-1}(\overline{\alpha}_{n-1}-\alpha_{n-2})^{-1}Y_{n-1,j}\notag\\
&-(\overline{\alpha}_{n}-\alpha_{n-1})^{-1}(\overline{\alpha}_{n-1}-\alpha_{n-2})^{-1}Y_{n-1,j}\beta_{j+1}\notag\\
&+(\overline{\alpha}_{n}-\alpha_{n-1})^{-1}(\overline{\alpha}_{n-1}-\alpha_{n-2})^{-1}\left[Y_{n-3,j+1}-Y_{n-2,j}\right].
\notag
\end{align}
We can continue iterating by invoking the formula \eqref{5.14} for $k=n-3$. After $j$ iterations, we will get to 
the difference $\left[Y_{n-j-1,j+1}-Y_{n-j,j}\right]$ of two entries from the counter-diagonal $\mathcal D_{n+1}(Y)$
which are zeros, by \eqref{5.3}. Thus, $j$ iterations of formula \eqref{5.14} (for $k=n-1,\ldots,n-j$) results in
\begin{align}
Y_{n-1,j+1}-Y_{n,j}=&\alpha_n\left(\sum_{\ell=1}^j \left(\prod_{i=1}^{\substack{\curvearrowright \\ \ell}}
(\overline{\alpha}_{n-i+1}-\alpha_{n-i})^{-1}\right)Y_{n-\ell+1,j}\right)\notag\\
&-\left(\sum_{\ell=1}^j \left(\prod_{i=1}^{\substack{\curvearrowright \\ \ell}}
(\overline{\alpha}_{n-i+1}-\alpha_{n-i})^{-1}\right)Y_{n-\ell+1,j}\right)\beta_{j+1}.\label{5.15}
\end{align}
We now combine latter formula with formulas \eqref{5.8}, \eqref{5.5} for $\Delta_{n,j+1}$ and $Y_{n,j+1}$, 
respectively, to compute
\begin{equation}
\Delta_{n,j+1}=\alpha_nY_{n,j+1}-Y_{n,j+1}\beta_j+Y_{n-1,j+1}-Y_{n,j}=\alpha_n\mu_{j+1}-\mu_{\j+1}\beta_{j+1}=0,
\label{5.16}
\end{equation}
where the last equality is the consequence of \eqref{5.2}. Due to induction hypothesis and \eqref{5.16}, me may apply 
implication \eqref{5.10} to conclude that $\Delta_{k,i}=0$
also for all $k,i$ such that $k+i= r+1$. By induction principle, $\Delta_{k,j}=0$ for all $k=1,\ldots,n$, $j=1,\ldots,m$
and since $\Delta=[\Delta_{k,j}]=AY-YB$ (by definition \eqref{5.8}), we conclude that $Y$ is a solution to \eqref{5.1}.

\smallskip

It remains to show that any $Y$ solving \eqref{5.1} is necessarily of the form \eqref{5.2}--\eqref{5.5}.
By Theorem \ref{T:4.4} (part (2)), the entries $Y_{k,j}$
(for $k+j\le n$) are the same for any solution $Y$ to the equation \eqref{4.42}. Since the zero
matrix is a solution, it then follows that $Y_{k,j}=0$ whenever $k+j\le n$. Equating the entries
from the bottom row in \eqref{5.1}, we get, on account of \eqref{4.1},
\begin{align}
&\alpha_nY_{n,1}-Y_{n,1}\beta_1=0,\label{5.17}\\
&\alpha_nY_{n,j}-Y_{n,j}\beta_{j}+Y_{n-1,j}-Y_{n,j-1}=0 \quad 
(j=2,\ldots,m).\label{5.18}
\end{align}
Equation \eqref{5.17} simply means that $Y_{n,1}=mu_1$ can be picked arbitrarily in 
$\Pi_{\alpha_n,\beta_1}$, the solution set of the scalar Sylvester equation \eqref{5.17}. 
By (the proof of) Theorem \ref{T:4.4}, any solution $Y$ to the equation \eqref{5.1} also satisfies
\begin{equation}
\widetilde{A}Y-Y\widetilde{B}=0,
\label{5.19}
\end{equation}
where $\widetilde{A}$ and $\widetilde{B}$ are given in \eqref{3.4c} and \eqref{4.32}.
Equating the $(n-1,2)$ entries in \eqref{5.19} we get 
$(\overline{\alpha}_n-\alpha_{n-1})Y_{n-1,2}-Y_{n,1}(\overline{\beta}_{2}-\beta_{1})=0$,
which is equivalent to \eqref{5.6} and thus, verifies \eqref{5.4} for $k=n-1$ and $j=2$.
Then we can apply formula \eqref{5.14} for $k=n-1$ and $j=1$ to get (we recall that $Y_{n-2,2}=Y_{n-1,1}=0$)
$$
Y_{n-1,2}-Y_{n,1}=\alpha_n(\overline{\alpha}_n-\alpha_{n-1})^{-1}Y_{n,1}-(\overline{\alpha}_n-\alpha_{n-1})Y_{n,1}\beta_2,
$$
which, being substituted into equation \eqref{5.18} (for $j=1$) leads us to 
\begin{align*}
0=&\alpha_nY_{n,2}-Y_{n,2}\beta_{2}+Y_{n-1,2}-Y_{n,1}\\
=&\alpha_n\left(Y_{n,2}+(\overline{\alpha}_n-\alpha_{n-1})^{-1}Y_{n,1}\right)-
\left(Y_{n,2}+(\overline{\alpha}_n-\alpha_{n-1})^{-1}Y_{n,1}\right)\beta_2.
\end{align*}
The latter equality tells us that the element 
$$
\mu_2:=Y_{n,2}+(\overline{\alpha}_n-\alpha_{n-1})^{-1}Y_{n,1} 
$$
must satisfy equation \eqref{5.2} for $j=2$, which verifies formula \eqref{5.5} for $j=1$.
So far, we verified formulas \eqref{5.4} and \eqref{5.5} for the entries from the two leftmost columns of $Y$.
Let us assume the formulas hold true for all entries from the $r\ge 2$ leftmost columns and let us write 
explicitly conditions $[\widetilde{A}Y-Y\widetilde{B}]_{k+1,r+1}=0$ (see \eqref{5.19}) for $k=1,\ldots,n-1$;
in view of \eqref{3.4c} and \eqref{4.32}, we have
$$
(\alpha_k-\overline{\alpha}_{k+1})Y_{k,r+1}+Y_{k-1,r+1}-Y_{k+1,r}(\beta_{r}-\overline{\beta}_{r+1})-
Y_{k+1,r-1}=0
$$
which being solved for $Y_{k,r+1}$ verifies formulas \eqref{5.4} for $j=r+1$ and $k\le n-1$. Since 
the formulas \eqref{5.4} and \eqref{5.5} hold for $j=r$ and all $k\le n$ by the assumption, equality
\eqref{5.15} holds for $j=r$. Substituting this equality into equation \eqref{5.18} (for $j=r+1$) gives
\begin{align*}
0=&\alpha_nY_{n,r+1}-Y_{n,r+1}\beta_{r+1}+Y_{n-1,r+1}-Y_{n,r}\\
=&\alpha_n\left(Y_{n,r+1}+\sum_{\ell=1}^r \left(\prod_{i=1}^{\substack{\curvearrowright \\ \ell}}
(\overline{\alpha}_{n-i+1}-\alpha_{n-i})^{-1}\right)Y_{n-\ell+1,r}\right)\notag\\
&-\left(Y_{n,r+1}+\sum_{\ell=1}^r \left(\prod_{i=1}^{\substack{\curvearrowright \\ \ell}}
(\overline{\alpha}_{n-i+1}-\alpha_{n-i})^{-1}\right)Y_{n-\ell+1,r}\right)\beta_{r+1}.
\end{align*}
We conclude that the element 
$$
\mu_{r+1}:=Y_{n,r+1}+\sum_{\ell=1}^r \left(\prod_{i=1}^{\substack{\curvearrowright \\ \ell}}
(\overline{\alpha}_{n-i+1}-\alpha_{n-i})^{-1}\right)Y_{n-\ell+1,r}
$$
must be in $\Pi_{\alpha_n,\beta_{r+1}}$ which verifies \eqref{5.5} for $j=r+1$. By induction principle, 
formulas \eqref{5.4} and \eqref{5.5} hold for all $j=1,\ldots,m$.
\end{proof}
\begin{corollary}
Let $A=\mathcal I_n(\alpha)$ and $B=\mathcal I_m(\beta)$ ($n\ge m$)
be Jordan blocks based on the elements $\alpha\sim\beta$. A matrix $Y=[Y_{k,j}]\in\bH^{n\times m}$  
satisfies $AY=YB$ if and only if 
\begin{equation}
Y_{k,j}=0 \; \; (2\le k+j\le n)\quad\mbox{and}\quad Y_{k,j}=\mu_{k+j-n} \; \; (n+1\le k+j\le n+m),
\label{5.20}   
\end{equation} 
where $\mu_1,\ldots,\mu_m\in\bH$ are any elements from the plane $\Pi_{\alpha,\beta}$
(see \eqref{4.4} and \eqref{4.5}), i.e.,
\begin{equation}
\alpha \mu_i=\mu_i\beta\quad\mbox{for}\quad i=1,\ldots,m.
\label{5.21}
\end{equation}
\label{C:5.2}
\end{corollary}
\begin{proof}
It suffices to let $\alpha_k=\alpha$ and $\beta_j=\beta$ in Theorem \ref{T:5.1} and to show that 
formulas \eqref{5.3}--\eqref{5.5} amount to \eqref{5.20} in this particular setting. This is clearly true
for $k+1\le n$ and for $Y_{n,1}=\mu_1$. For the rest, we will use equalities
\begin{equation}
\overline{\alpha} \mu_i=\mu_i\overline{\beta} \quad (i=1,\ldots,m),
\label{5.22}
\end{equation}
which follow from \eqref{5.21} since ${\rm Re}(\alpha)={\rm Re}(\beta)$, by \eqref{2.1}.
Due to equalities \eqref{5.21} and \eqref{5.22}, formula \eqref{5.5a} takes the form
$$
Y_{n-j,j+1}=(\overline{\alpha}-\alpha)^{-j}\mu_1(\overline{\beta}-\beta)^j=\mu_1\quad (j=1,\ldots,m-1).
$$
By \eqref{5.7}, $Y_{n,2}$ is necessarily of the form $Y_{n,2}=\mu_2-(\overline{\alpha}-\alpha)^{-1}\mu_1$
where $\mu_2$ is any element in $\Pi_{\alpha,\beta}$. But since 
$$
\alpha(\overline{\alpha}-\alpha)^{-1}\mu_1=(\overline{\alpha}-\alpha)^{-1}\alpha 
(\overline{\alpha}-\alpha)^{-1}\mu_1\beta,
$$
i.e., since $(\overline{\alpha}-\alpha)^{-1}\mu_1\in\Pi_{\alpha,\beta}$, it follows that 
$Y_{n,2}$ can be chosen arbitrarily in $\Pi_{\alpha,\beta}$. In other words, we may let 
$Y_{n,2}=\mu_2\in\Pi_{\alpha,\beta}$. Since $Y_{n-j,j+1}=\mu_1$ for all $j=0,\ldots,m-1$, the formula 
\eqref{5.4} (for $k+j=n+2$) now gives
$$
Y_{n-j+1,j+1}=(\overline{\alpha}-\alpha)^{-j}\mu_2(\overline{\beta}-\beta)^j=\mu_2\quad (j=2,\ldots,m-1)
$$
verifying the Hankel structure of $\mathcal D_{n+1}(Y)$. Now we have by \eqref{5.5},
$$
Y_{n,3}=\mu_3-(\overline{\alpha}-\alpha)^{-1}\mu_2-(\overline{\alpha}-\alpha)^{-2}\mu_1,
$$
and since the second and the third terms on the right belong to $\Pi_{\alpha,\beta}$ and 
$\mu_3$ can be chosen in $\Pi_{\alpha,\beta}$ arbitrarily, we may let $Y_{n,3}=\mu_3\in\Pi_{\alpha,\beta}$.
Now we use formula \eqref{5.4} and take into account the Hankel structure of $\mathcal D_{n+1}(Y)$ to verify
the Hankel structure of $\mathcal D_{n+2}(Y)$. Repeating this argument $m$ times we arrive at the desired 
conclusion. 
\end{proof}
Recall that all results in Sections 4 and 5 were obtained under assumption that $n\ge m$. 
The assumption is not restrictive in the following sense: if $m\ge n$, then one can apply 
the obtained results to the adjoint equation \eqref{3.18a} as in the proof of Theorem \ref{T:3.7},
i.e., to replace $A$, $B$ and $C$ by $B^*$, $A^*$ and $C^*$, respectively. We omit further details.

\bibliographystyle{amsplain}

\begin{thebibliography}{10}


\bibitem{bol}
V.~Bolotnikov, {\em Polynomial interpolation over quaternions}, J. Math. Anal. Appl. {\bf 421} 
(2015), no. 1, 567--590.

\bibitem{bolalg}   
V.~Bolotnikov, {\em Zeros and factorizations of quaternion polynomials: the algorithmic approach}, 
Preprint.

\bibitem{bolalg1}
V.~Bolotnikov, {\em Quasi-ideals in the ring of quaternion polynomials}, Preprint.

\bibitem{brenner}
J.~L.~Brenner, {\em Matrices of quaternions}, Pacific J. Math., {\bf 1} 
(1951), 329--335.


\bibitem{ces}
F.~Cecioni, {\em Sulle equazioni fra matrici $AX=XB$, $X^n=A$},
R. Acad. dei Lincei, Rend. {\bf 18} (1909), 566--571.

\bibitem{ces1}
F.~Cecioni, {\em Sopra alcune operazioni algebriche sulle matrici}, 
Ann. R. Scuola Norm. Sup. Pisa {\bf 11} (1910), 1--40.



\bibitem{fro} G. Frobenius, {\em  \"Uber die mit einer Matrix vertauschbaren Matrizen}, 
Sitzungsber. Preuss. Akad. Wiss., 1910, 3--15.

\bibitem{blr}
I.~Gohberg, P.~Lancaster, and L.~Rodman, {\em Matrix polynomials}, 
Academic Press, , New York-London, 1982.

\bibitem{huang}
L.~Huang, {\em The matrix equation $AXB-CXD=E$ over the quaternion field}, Linear Algebra Appl. 
{\bf 234} (1996), 197--208.

\bibitem{jameson}
A. Jameson, {\em Solutions of the equation $AX + XB = C$ by inverse of an $M\times M$ or 
$N\times N$ matrix}, SIAM J. Appl. Math. {\bf 16} (1968) 1021--1022.

\bibitem{lanc}
P.~Lancaster, {\em Explicit solutions of linear matrix equations}, SIAM Rev. {\bf 12} 1970 544--566.

\bibitem{lee}
H.~C.~Lee, {\em  Eigenvalues and canonical forms of matrices with 
quaternion coefficients}, Proc. Roy. Irish Acad. {\bf 52} (1949), 253--260.



\bibitem{ma}
E.~C.~Ma, {\em A finite series solution of the matrix equation $AX-XB=C$,}
SIAM J. Appl. Math. {\bf 14} (1966) 490--495.

\bibitem{rosen}
M.~Rosenblum, {\em On the operator equation $BX-XA=Q$}, Duke Math. J. {\bf 23} (1956) 263±270.


\bibitem{rutherford}
D.~E.~Rutherford, {\em On the solution of the matrix equation $AX+XB=C$}, Proc. Akad. Wet. 
Amsterdam {\bf 35} (1932), 54--59.

\bibitem{song}
C.~Song, and G.~Chen, {\em On solutions of matrix equation $XF-AX=C$ and $XF-A\widetilde{X}=C$ over quaternion 
fie}, J. Appl. Math. Comput. {\bf 37} (2011), no. 1-2, 57--68.




\bibitem{sylv}
J.~J.~Sylvester, {\em Sur l'equation en matrices $px=xq$}, C. R. Acad. Sci. Paris {\bf 99} (1884) 
67--71, 115--116.

\bibitem{tait}
P.~G.~Tait, {\em An elementary treatise on quaternions},  Oxford, Clarendon Press, 1867.


\bibitem{wed}
J.H.M.~Wedderburn, {\em On division algebras}, Trans. Amer. Math. Soc. {\bf 22} (1921), 
129--135.


\bibitem{wieg}
N.Wiegmann, {\em Some theorems on matrices with real quaternion elements}, 
Canad. J. Math. {\bf 7} (1955) 191--201.




\end{thebibliography}

\end{document}